\documentclass[english]{amsart}

\usepackage{esint}
\usepackage[svgnames]{xcolor} 
\usepackage[colorlinks,citecolor=red,pagebackref,hypertexnames=false,breaklinks]{hyperref}
\usepackage{pgf,tikz}
\usepackage{pdfsync}

\usepackage{dsfont}
\usepackage{url}
\usepackage[utf8]{inputenc}
\usepackage[T1]{fontenc}
\usepackage{lmodern}
\usepackage{babel}
\usepackage{mathtools}  
\usepackage{amssymb}
\usepackage{lipsum}
\usepackage{mathrsfs}
\usepackage{color}

\newtheorem{theorem}{Theorem}[section]
\newtheorem{proposition}{Proposition}[section]
\newtheorem{lemma}{Lemma}[section]

\newtheorem{corollary}{Corollary}[section]

\newtheorem{remark}{Remark}[section]

\numberwithin{equation}{section}

\title[Global logarithmic global stability of a Cauchy problem]{Global logarithmic stability of  a Cauchy problem for anisotropic wave equations}

\author[Mourad Bellasoued]{Mourad Bellassoued}
\address{University Tunis El Manar, ENIT-LAMSIN, BP 37, 1002, Tunis, Tunisia}
\email{mourad.bellassoued@enit.utm.tn}

\author[Mourad Choulli]{Mourad Choulli}
\address{Universit\'e de Lorraine, 34 cours L\'eopold, 54052 Nancy cedex, France}
\email{mourad.choulli@univ-lorraine.fr}

\thanks{MC is supported by the grant ANR-17-CE40-0029 of the French National Research Agency ANR (project MultiOnde). }

\date{}

\begin{document}

\begin{abstract}
We discuss  a Cauchy problem for anisotropic wave equations. Precisely, we address the question to know which kind of Cauchy data on the lateral boundary  are necessary  to guarantee the uniqueness of continuation of solutions of an anisotropic wave equation. In the case where the uniqueness holds, the natural problem that arise naturally in this context is to estimate the solutions, in some appropriate space, in terms of norms of  the Cauchy data. 

We aim in this paper to convert, via a reduced Fourrier-Bros-Iagolnitzer transform, the known stability of the Cauchy problem for anisotropic elliptic equations to stability of a Cauchy problem for anisotropic waves equations. By proceeding in that way one the main difficulties is to control the residual terms, induced by the reduced Fourrier-Bros-Iagolnitzer transform, by a Cauchy data. Also,  a uniqueness of continuation result, from Cauchy data, is obtained as byproduct of stability results.
\end{abstract}



\maketitle

\tableofcontents

\section{Introduction}\label{introduction}

\subsection{What is our objective ?}

Let $\Omega$ be a Lipschitz bounded domain of $\mathbb{R}^n$, $n\ge 2$ and $T>0$. Consider then the anisotropic wave equation with matrix coefficients $A(x)=(a^{ij}(x))$,
\begin{equation}\label{0.1}
Pu=\mbox{div}(A\nabla u)-\partial_t^2u=0 \quad \mbox{in}\; \Omega \times (-T,T).
\end{equation}
Suppose that $A\in C^1(\overline{\Omega},\mathbb{R}^{n\times n})$ satisfies, for some $\kappa >1$,
\[
\kappa^{-1} |\xi|^2\le A(x)\xi \cdot \xi \le \kappa |\xi|^2,\quad x\in \overline{\Omega}\; \mathrm{and}\; \xi \in \mathbb{R}^n.
\]
Let $\omega \Subset \Omega$ and $d=\mathrm{diam}(\Omega)$. From \cite[Theorem 1 in page 791]{Ro95} we have the following uniqueness of continuation result from an interior data : 

there exists  a constant $\varkappa=\varkappa (\kappa)>0$ so that, if $T>\varkappa d$ and $u\in H^2 (\Omega \times (-T,T))$ satisfies 
\[
Pu=0\; \mathrm{in}\; \Omega \times (-T,T)\quad \mathrm{and}\quad  u=0\;  \mathrm{in}\; \omega \times (-T,T),
\]
then 
\[
u=0\; \mathrm{in}\;  \Omega \times (-(T-\varkappa d),T-\varkappa d), 
\]

Denote by  $\nu$  the unit normal outward vector field on $\partial \Omega$ and let $\Gamma$ be an nonempty open subset of $\partial \Omega$. 

Assume that the assumptions on $A$ hold in neighborhood of $\overline{\Omega}$. From the preceding result one can get the following of uniqueness of continuation from Cauchy data:  

if $T>\varkappa d$ and $u\in H^2 (\Omega \times (-T,T))$ satisfies 
\[
Pu=0\; \mathrm{in}\; \Omega \times (-T,T)\quad \mathrm{and}\quad  u=\partial_\nu u=0\;  \mathrm{on}\; \Gamma \times (-T,T),
\]
then 
\[
u=0\; \mathrm{in}\;  \Omega \times (-(T-\varkappa d),T-\varkappa d). 
\]

In the present work we address the issue of quantifying such kind of unique continuation results. The fact that the uniqueness of continuation does not hold in the whole time domain renders the analysis quite complicated as we will see in the sequel.

\subsection {Main result.}

Prior to state our main result we introduce some definitions and notations. 

If $K$ is the closure of an nonempty open bounded subset of $\mathbb{R}^n$, $n\ge 2$, define, where $0< \alpha \le 1$, the semi-norm $[f]_\alpha$ by
\[
[f]_\alpha =\sup \left\{ \frac{|f(x)-f(y)|}{|x-y|^\alpha};\; x,y\in K,\; x\ne y\right\}.
\]
The  usual H\"older space 
\[
C^{0,\alpha} (K)=\{f\in C^0(K);\; [f]_\alpha <\infty\},
\] 
 endowed with the norm
\[
\|f\|_{C^{0,\alpha} (K)}=\|f\|_{L^\infty (K)}+[f]_\alpha ,
\]
is a Banach space.

Also, recall that
\[
C^{1,\alpha} (K)=\{f\in C^1(K);\; \partial_{x_j}f\in C^{0,\alpha} (K)  ,\; 1\le j\le n\},
\]
 is a Banach space when it is equipped with its natural norm
\[
\|f\|_{C^{1,\alpha} (K)}=\|f\|_{L^\infty (K)}+\sum_{j=1}^n\left(\| \partial_{x_j}f\|_{L^\infty (K)}+[\partial_{x_j}f]_\alpha\right).
\]

Let $\Omega$ be a $C^\infty$ bounded domain of $\mathbb{R}^n$, $n\ge 2$. The following notations will be used in the rest of this paper 
\begin{align*}
&I_s=(-s,s),\quad s>0,
\\
&Q_s=\Omega \times I_s,\quad s>0,
\\
&J_\delta =\{ t\in \mathbb{R};\; (1-\delta )T<|t|<T\},\quad 0<\delta <1,
\\
& R_\delta =\Omega \times J_\delta ,\quad 0<\delta <1,
\\
&\Sigma_s=\Gamma \times I_s, \quad s>0,
\end{align*} 
where $\Gamma$ is a given nonempty open subset of $\partial \Omega$.

Denote by $P$ the anisotropic wave operator acting on $Q_T$ as follows 
\[
Pu=\mbox{div}(A\nabla u )-\partial_t^2u,
\]
where $A=(a^{ij})$ is a symmetric matrix with  coefficients in $W^{1,\infty}(\Omega )$ satisfying 
\begin{equation}\label{main1}
\kappa ^{-1}|\xi |^2 \le A(x)\xi \cdot \xi \le \kappa |\xi |^2,\;\;  x\in \Omega , \; \xi \in \mathbb{R}^n,
\end{equation}
and
\begin{equation}\label{main2}
\sum_{k=1}^n\left|\sum_{i,j=1}^n\partial_k a^{ij}(x)\xi _i\xi_j\right| \le \varkappa |\xi|^2,\quad x\in \Omega ,\; \xi \in \mathbb{R}^n,
\end{equation}
for some constants $\kappa \ge 1$ and $\varkappa>0$.

Let
\begin{align*}
&\mathcal{X}_T=C^{1,\alpha}(\overline{Q}_T)\cap H^1(I_T,H^1(\Omega )),
\\
&\mathcal{Y}_T=H^2(I_T,H^2(\Omega ))\cap H^3(I_T,H^1(\Omega ))\cap \mathcal{X}_T.
\end{align*}

These spaces are endowed with their natural norms as intersections of Banach spaces. We recall that if $\mathcal{X}=\mathcal{X}_0\cap \mathcal{X}_1$ is the intersection of two Banach spaces then we usually equip $\mathcal{X}$ with the norm
\[
\|x\|_\mathcal{X}=\|x\|_{\mathcal{X}_0}+\|x\|_{\mathcal{X}_1}.
\]

In the sequel by $H^\ell (\Gamma)$, $\ell\ge 0$, we mean the quotient space
\[
H^\ell (\Gamma )=\{\phi =\psi_{|\Gamma};\; \psi \in H^\ell (\partial \Omega )\},
\]
that we equip with its natural quotient norm
\[
\|\phi\|_{H^\ell (\Gamma )}=\min\{ \|\psi\|_{H^\ell (\partial \Omega)},\; \psi \in H^\ell (\partial \Omega )\; \mbox{such that}\; \psi_{|\Gamma}=\phi\} .
\]

Following Lions-Magenes's notations, we set
\[
H^{1,1}(\Sigma _T)=L^2(I_T,H^1(\Gamma ))\cap H^1(I_T,L^2(\Gamma )).
\]

For $u\in \mathcal{Y}_T$, set
\begin{align*}
&\mathcal{N}(u)=\|u\|_{\mathcal{Y}_T},
\\
&\mathcal{D}(u)=\|Pu\|_{H^1(I_T,L^2(\Omega ))}+\|u\|_{H^{1,1}(\Sigma _T)}+\|\partial_\nu u\|_{L^2(\Sigma_T)},
\\
&\mathfrak{D}_{\delta}(u)=\|u\|_{H^3\left(J_{\delta/2},H^{3/2}(\partial \Omega )\right)}+\|\partial_\nu u\|_{H^1\left(J_{\delta/2}, L^2(\partial \Omega )\right)}.
\end{align*}

Let
\[
\mathfrak{d}=\sup_{y\in \Omega}\inf_{x\in \partial \Omega} d(x,y),
\]
with
\begin{align*}
d(x,y)=\inf\{L(\gamma ); \gamma :[0,1]&\rightarrow \overline{\Omega}
\\
&C^1\mbox{-piecewise path so that}\; \gamma (0)=x,\; \gamma (1)=y\},
\end{align*}
where $L(\gamma)$ denotes the length of $\gamma$.

For $k\ge 2$ an integer, define $\mathfrak{e}_k$  as follows
\[
\mathfrak{e}_k=\exp \circ \exp \circ \ldots \circ \exp \quad (k\mbox{-times}),
\]
where $\exp$ is the usual exponential function.

We aim in the present work to prove the following result.

\begin{theorem}\label{maintheorem}
Let $s\in (0,1/2)$, $T_0>\mathfrak{d}$, $\mathfrak{d}<\sqrt{8}\, \overline{\rho}<T_0$ and $0<\underline{\rho}<\overline{\rho}$. There exist two constants $C>0$ and $c>0$, only depending on $\Omega$, $T_0$, $\kappa$, $\varkappa$, $\underline{\rho}$, $\overline{\rho}$, $s$ and $\alpha$, so that, for any $T\ge T_0$, $\delta \in [\sqrt{8}\, \underline{\rho}/T,\sqrt{8}\, \overline{\rho}/T]$ and $u\in \mathcal{Y}_T$, we have
\begin{equation}\label{1++}
C\|u\|_{L^2\left(I_T,H^1(\Omega )\right)} \le \delta ^s\mathcal{N}(u) + \mathfrak{e}_3\left(c\delta^{-14}\right)\left(\mathcal{D}(u)+\mathfrak{D}_{\overline{\delta}}(u) \right),
\end{equation}
where $\overline{\delta}=\frac{\sqrt{8}\, \overline{\rho}}{T}$.
\end{theorem}

Theorem \ref{maintheorem} is to our knowledge the first stability result of a Cauchy problem for anisotropic wave equations. The presence of the  term $\mathfrak{D}_{\overline{\delta}}(u)$ in \eqref{1++} can be explained by the fact that the preceding uniqueness of continuation from Cauchy data does not hold in the whole time interval. In counterpart our stability inequality is global.

Quantifying the uniqueness of continuation from Cauchy data, as it is stated above, is a challenging issue.

We point out that a global logarithmic stability result of the Cauchy problem for an anisotropic heat equation was recently proved by second author and M. Yamamoto \cite{CY}. In that paper the problem is solved directly.

For both the heat and the wave equations, one of the main difficulties comes from the fact that the data at the end points of the time interval is not known, and therefore must be controlled by the Cauchy data. Another difficulty that arises in the case of the wave equation is induced by the residual terms related to the Fourrier-Bros-Iagolnitzer transform. These residual terms have to be controlled again by the Cauchy data.

The situation is very different in the case of anisotropic elliptic equations for which we have global (single) logarithmic stability result. This result is the best possible that can be obtained in the general case. Indeed, we know since Hadamard that the Cauchy problem for elliptic equations is ill-posed in the sense that there is no hope to get a Lipschitz (or even H\"older) stability estimate. Precisely, Hadamard \cite[page 33]{Ha} gave an example in which the stability is exactly logarithmic. Some references and comments concerning the elliptic Cauchy problem are provided in Section \ref{elliptic}.

It is worth mentioning that, under an additional ``smallness'' condition of the coefficient of the matrix $A$, we can improve Theorem \ref{maintheorem} (see Theorem \ref{theorem3.1} for more details).

\subsection{Outline.}

The rest of this text is organized as follows. Section 2 is mainly devoted to stability results of the Cauchy problem for anisotropic elliptic equations in a form adapted to our approach for studying Cauchy problems for wave equations. Precisely, we need a stability result for the elliptic Cauchy problem which is uniform with respect to one parameter family of domains. 
 
In Section 3, we give some results on the reduced Fourrier-Bros-Iagolnitzer transform that are necessary to convert the stability result of the Cauchy problem for elliptic equations to a stability result of a Cauchy problem for wave equations.  We prove Theorem \ref{maintheorem} in Section 4. We give in Section 5 comments concerning a variant of Theorem \ref{maintheorem}. We discuss in  Section 6 an improvement of the general stability result when we impose a ``smallness'' condition on the coefficients of $A$.

For sake of completness we added in Appendix \ref{app1} a self-contained proof of a stability result of an elliptic Cauchy problem. Our analysis relies essentially on three-ball inequalities, for both solutions and their gradient, that are obtained using a Carleman inequality proved in \cite{Ch2016}. We take this opportunity to improve and extend the results obtained by the second author in \cite{Ch2016}.

\section{Elliptic Cauchy problem for one parameter family of domains}\label{elliptic}

Let $D$ be a Lipschitz bounded domain of $\mathbb{R}^d$, $d\ge 2$,  and consider the elliptic operator $L$ acting on $D$ as follows 
\[
Lu=\mbox{div}(A\nabla u),
\]
where $A=(a^{ij})$ is a symmetric matrix with  coefficients in $W^{1,\infty}(D)$ satisfying
\begin{equation}\label{el1}
\kappa ^{-1}|\xi |^2 \le A(x)\xi \cdot \xi \le \kappa |\xi |^2,\;\;  x\in D , \; \xi \in \mathbb{R}^d,
\end{equation}
and
\begin{equation}\label{el2}
\sum_{k=1}^d\left|\sum_{i,j=1}^d\partial_k a^{ij}(x)\xi _i\xi_j\right| \le \varkappa |\xi|^2,\quad x\in D ,\; \xi \in \mathbb{R}^d,
\end{equation}
with some constants $\kappa \ge 1$ and $\varkappa>0$.

\begin{theorem}\label{theorem-el1}
Let $\mathcal{C}$ be a nonempty open subset of $\partial D$ and $0<\alpha \le 1$. There exist $C>0$, $c >0$ and $\beta >0$, only depending on $D$, $\kappa$, $\varkappa$, $\alpha$ and $\mathcal{C}$, so that, for any $u\in C^{1,\alpha}(\overline{D} )$ satisfying $Lu\in L^2(D)$ and $0<\epsilon <1$,  we have
\begin{equation}\label{el3}
C\|u\|_{H^1(D)}\le \epsilon ^\beta \|u\|_{C^{1,\alpha}(\overline{D} )}+e^{c/\epsilon}\left(\|u\|_{L^2(\mathcal{C})}+\|\nabla u\|_{L^2(\mathcal{C})}+\|Lu\|_{L^2(D)}\right).
\end{equation}
\end{theorem}

This theorem was already proved in \cite{Ch2016} under an additional geometric condition on the domain. We provide a detailed proof of Theorem \ref{theorem-el1} in Appendix \ref{app1}. There is a substantial literature on elliptic Cauchy problems in different forms with various assumptions. We just quote here the following few recent references \cite{ABRV, ARRV,Bo,BD,Ch2016, Ch2020,Ph} (see also \cite{Boul} for non stationary Stokes problem and \cite{Bo2,CY} for the parabolic case). 

In order to tackle the Cauchy problem for the wave equation with need a variant of Theorem \ref{theorem-el1} with uniform constants for a family of domains of the form $Q_\rho=\Omega \times (-\rho,\rho)$, where $\rho >0$  is real and $\Omega$ is a bounded Lipschitz domain $\mathbb{R}^{n-1}$, $n\ge 2$.

Fix $0<\underline{\rho}<\overline{\rho}$ and assume that \eqref{el1} and \eqref{el2} hold with $D=Q_{\overline{\rho}}$. 

We define $a^{ij}_\rho$, where $\rho \in [\underline{\rho},\overline{\rho}]$, as follows
\begin{align*}
&a^{ij}_\rho (x',x_n)=a_{ij}(x',\rho x_n),\quad (x',x_n)\in Q_1,\; 1\le i,j\le n-1,
\\
&a^{in}_\rho (x',x_n)=\frac{1}{\rho}a_{in}(x',\rho x_n),\quad (x',x_n)\in Q_1,\; 1\le i\le n-1,
\\
&a^{nn}_\rho (x',x_n)=\frac{1}{\rho^2}a_{nn}(x',\rho x_n),\quad (x',x_n)\in Q_1.
\end{align*}

Set $A_\rho =(a^{ij}_\rho)$. In light of \eqref{el1} and \eqref{el2} for $D=Q_{\overline{\rho}}$, we have
\begin{equation}\label{el4}
\kappa ^{-1}\min(1,\overline{\rho}^{-2})|\xi |^2 \le A_\rho(x)\xi \cdot \xi \le \kappa\max(1,\underline{\rho}^{-2}) |\xi |^2,\;\;  x\in Q_1 , \; \xi \in \mathbb{R}^n,
\end{equation}
and
\begin{equation}\label{el5}
\sum_{k=1}^n\left|\sum_{i,j=1}^n\partial_k a_\rho^{ij}(x)\xi _i\xi_j\right| \le \varkappa \max(1, \underline{\rho}^{-2})|\xi|^2,\quad x\in Q_1 ,\; \xi \in \mathbb{R}^n.
\end{equation}

The proof of \eqref{el5} is obvious, while \eqref{el4} follows readily by noting that
\[
A_\rho(x)\xi \cdot \xi= A(x',\rho x_n)(\xi ',\rho^{-1}\xi_n)\cdot (\xi ',\rho^{-1}\xi_n).
\]

Define $L_\rho$ the elliptic operator acting on $Q_\rho$ as follows 
\[
L_\rho v=\mbox{div}(A_\rho \nabla v).
\]

Let $u\in C^\infty(\overline{Q}_\rho)$ and set $v(x',x_n)=u(x',\rho x_n)$, $(x',x_n)\in Q_1$. Then straightforward computations yield
\begin{align}
&\|v\|_{L^2(Q_1)}=\rho^{-1/2}\|u\|_{L^2(Q_\rho)},\label{el6}
\\
&\|\nabla v\|_{L^2(Q_1)}\ge \rho^{-1/2}\min(1,\rho)\|\nabla u\|_{L^2(Q_\rho)},\label{el7}
\\
&\|L_\rho v\|_{L^2(Q_1)}=\rho^{-1/2}\|Lu\|_{L^2(Q_\rho)}\label{el8}
\end{align}

Similarly, if $\Gamma$ is a nonempty open subset of $\partial \Omega$ and $\Sigma_\rho=\Gamma \times (-\rho,\rho)$ then we have
\begin{align}
&\|v\|_{L^2(\Sigma_1)}=\rho^{-1/2}\|u\|_{L^2(\Sigma_\rho)},\label{el9}
\\
&\|\nabla v\|_{L^2(\Sigma_1)}\le \rho^{-1/2}\max(1,\rho)\|\nabla u\|_{L^2(\Sigma_\rho)}\label{el10}.
\end{align}

In light of identities and inequalities \eqref{el6} to \eqref{el10}, we get by applying Theorem \ref{theorem-el1}

\begin{theorem}\label{theorem-el2}
Let $\Gamma$ be a nonempty open subset of $\partial \Omega$ and $0<\alpha \le 1$. Fix $0<\underline{\rho}<\overline{\rho}$. Then there exist $C>0$, $c >0$ and $\beta >0$, only depending on $D$, $\kappa$, $\varkappa$, $\alpha$, $\Gamma$, $\underline{\rho}$ and $\overline{\rho}$ so that,  for any $\rho \in [\underline{\rho},\overline{\rho}]$ and $u\in C^{1,\alpha}(\overline{Q}_\rho )$ satisfying $Lu\in L^2(Q_\rho)$,  and $0<\epsilon <1$,  we have
\begin{align}
C\|u\|_{H^1(Q_\rho)}\le \epsilon ^\beta \|&u\|_{C^{1,\alpha}(\overline{Q}_\rho )}\label{el3}
\\ 
&+e^{c/\epsilon}\left(\|u\|_{L^2(\Sigma_\rho)}+\|\nabla u\|_{L^2(\Sigma_\rho)}+\|Lu\|_{L^2(Q_\rho)}\right).\nonumber
\end{align}
\end{theorem}

\section{Fourier-Bros-Iagolnitzer transform}

The Fourier-Bros-Iagolnitzer transform is a useful tool to transfer results from an elliptic equation to a wave equation. This transform is well known in control theory community (see for instance \cite{LR,Ro91,Ro95}). In the present section we demonstrate specific results on this transform that are necessary to establish our stability results of Cauchy problems for waves equations.

Fix $0<\delta <1$ and  $t_0\in I_{(1-\delta)T}$, and pick $\chi \in C^\infty _0(I_T)$ so that $0\le \chi \le 1$, $\chi =1$ in $I_{(1-\delta/2)T}$, and
\[
|\chi '| \le \varpi (\delta T)^{-1},\quad |\chi ''| \le \varpi (\delta T)^{-2},
\]
for some universal constant $\varpi$.

Define, for $(x,\tau )\in Q_{\delta T/\sqrt{8}}$ (resp. $(x,\tau )\in \Sigma_{\delta T/\sqrt{8}}$) and $\lambda >0$, the reduced Fourier-Bros-Iagolnitzer transform as follows
\begin{align*}
[\mathscr{F}_{\lambda ,t_0} f](x,\tau )&=\int_{\mathbb{R}}e^{-\lambda(i\tau -(t-t_0))^2/2}\chi (t)f(x,t)dt
\\
&=e^{\lambda \tau ^2/2}\int_{\mathbb{R}}e^{-\lambda [(t-t_0)^2-2i\tau (t-t_0)]/2}\chi (t)f(x,t)dt.
\end{align*}

\begin{lemma}\label{lemma3.1}
$\mathrm (i)$ If $f\in C_0^\infty (Q_T)$ then
\begin{equation}\label{e1}
\|\mathscr{F}_{\lambda ,t_0} f\|_{L^2(Q_{\delta T/\sqrt{8}})}\le 2^{1/4}Te^{\lambda (\delta T)^2/16}\|f\|_{L^2(Q_T)}.
\end{equation}
Thus, $\mathscr{F}_{\lambda ,t_0}$ is extended as a bounded operator from $L^2(Q_T)$ into $L^2(Q_{\delta T/\sqrt{8}})$.
\\
$\mathrm (ii)$ Let $f\in C_0^\infty (\overline{\Sigma}_T)$. Then
\begin{equation}\label{e1+}
\|\mathscr{F}_{\lambda ,t_0} f\|_{L^2(\Sigma_{\delta T/\sqrt{8}})}\le 2^{1/4}Te^{\lambda (\delta T)^2/16}\|f\|_{L^2(\Sigma_T)}.
\end{equation}
Therefore, $\mathscr{F}_{\lambda ,t_0}$ is extended as a bounded operator from $L^2(\Sigma_T)$ into $L^2(\Sigma_{\delta T/\sqrt{8}})$.
\end{lemma}

\begin{proof}
$\mathrm (i)$ For $f\in C_0^\infty (Q_T)$, we have
\[
|[\mathscr{F}_{\lambda ,t_0} f](x,\tau )|\le e^{\lambda \tau ^2/2}\int_{\mathbb{R}}e^{-\lambda(t-t_0)^2/2}\chi (t)|f(x,t)|dt.
\]
As
\[
\int_{\mathbb{R}}e^{-\lambda(t-t_0)^2/2}\chi (t)|f(x,t)|dt\le \int_{-T}^{T}e^{-\lambda(t-t_0)^2/2}|f(x,t)|dt\le \int_{-T}^T|f(x,t)|dt,
\]
we get by applying Cauchy-Schwarz's inequality
\[
|[\mathscr{F}_{\lambda ,t_0} f](x,\tau )|^2\le 2Te^{\lambda \tau ^2}\int_{-T}^T|f(x,t)|^2dt.
\]
Whence
\[
\|\mathscr{F}_{\lambda ,t_0} f\|_{L^2(Q_{\delta T/\sqrt{8}})}\le \frac{\sqrt{\delta T}}{2^{1/4}}\sqrt{2T}e^{\lambda (\delta T)^2/16}\|f\|_{L^2(Q_T)}
\]
as expected.
\\
$\mathrm (ii)$ The proof is quite similar to that of (i).
\end{proof}

\begin{lemma}\label{lemma3.2}
$\mathrm (i)$ Let $f\in H^1(I_T,L^2(\Omega ))$, we have
\begin{equation}\label{e2+}
\partial _\tau \mathscr{F}_{\lambda ,t_0} f=i\mathscr{F}_{\lambda ,t_0} (\partial _tf)+k_{\lambda ,t_0} ,
\end{equation}
the function $k_{\lambda ,t_0} \in L^2(Q_{\delta T/\sqrt{8}})$ satisfies
\begin{equation}\label{e3+}
\|k_{\lambda ,t_0} \|_{L^2(Q_{\delta T/\sqrt{8}})}\le  \varpi   e^{-\lambda (\delta T)^2/16}\|f\|_{L^2(R_{\delta/2})}.
\end{equation}
$\mathrm (ii)$ For $f\in H^2(I_T,L^2(\Omega ))$, we have
\begin{equation}\label{e2}
\partial _\tau ^2\mathscr{F}_{\lambda ,t_0} f=-\mathscr{F}_{\lambda ,t_0} (\partial _t^2f)+g_{\lambda ,t_0} ,
\end{equation}
the function $g_{\lambda ,t_0} \in L^2(Q_{\delta T/\sqrt{8}})$ satisfies
\begin{equation}\label{e3}
\|g_{\lambda ,t_0} \|_{L^2(Q_{\delta T/\sqrt{8}})}\le 2\varpi (2T+1)(\delta T)^{-1}  e^{-\lambda (\delta T)^2/16}\|f\|_{H^1(J_{\delta /2},L^2(\Omega))}.
\end{equation}
\end{lemma}

\begin{proof}
We prove $\mathrm (ii)$. The proof of $\mathrm (i)$ can be easily deduced from that of $\mathrm(ii) $.

By density, it is sufficient to prove \eqref{e2} and \eqref{e3} when $f\in C_0^\infty (Q_T)$. In light of the relation
\[
\partial _\tau e^{-\lambda(i\tau -(t-t_0))^2/2}=-i\partial _te^{-\lambda(i\tau -(t-t_0))^2/2},
\]
we get
\[
\partial _\tau [\mathscr{F}_{\lambda ,t_0} f](x,\tau )=-i\int_{\mathbb{R}}\partial_t e^{-\lambda(i\tau -(t-t_0))^2/2}\chi (t)f(x,t)dt.
\]
We obtain by making an integration by parts in the right hand side
\[
\partial _\tau [\mathscr{F}_{\lambda ,t_0} f](x,\tau )=i\int_{\mathbb{R}}e^{-\lambda(i\tau -(t-t_0))^2/2}\partial _t[\chi (t)f(x,t)]dt.
\]
An iteration of this formula then yields
\[
\partial _\tau ^2[\mathscr{F}_{\lambda ,t_0} f](x,\tau )=-\int_{\mathbb{R}}e^{-\lambda(i\tau -(t-t_0))^2/2}\partial _t^2[\chi (t)f(x,t)]dt.
\]
In consequence
\[
\partial _\tau ^2\mathscr{F}_{\lambda ,t_0} f=-\mathscr{F}_{\lambda ,t_0} (\partial _t^2f)+g_{\lambda ,t_0} ,
\]
where
\[
g_{\lambda ,t_0} (x,\tau )=-\int_{\mathbb{R}}e^{-\lambda(i\tau -(t-t_0))^2/2}[2\chi' (t)\partial _tf(x,t)+\chi ''(t)f(x,t)]dt.
\]
We have  similarly
\[
|g_{\lambda ,t_0} (x,\tau )|\le e^{\lambda \tau ^2/2}\int_{\mathbb{R}}e^{-\lambda(t-t_0)^2/2}|2\chi' (t)\partial _tf(x,t)+\chi ''(t)f(x,t)|dt.
\]
But 
\[
\mbox{supp}\,\chi ',\; \mbox{supp}\, \chi'' \subset J_{\delta T/2}
\]
and
\[
|t-t_0|\ge \delta T/2,\quad \mbox{if}\; t\in J_{\delta T/2}.
\]
Hence
\begin{equation}\label{2.1}
|g_{\lambda ,t_0} (x,\tau )|\le \varpi e^{\lambda \tau ^2/2}(2T+1)(\delta T)^{-2}e^{-\lambda (\delta T)^2/8}\int_{J_{\delta/2}}[|\partial _tf(x,t)|+|f(x,t)|]dt.
\end{equation}

We proceed as in the proof of Lemma \ref{lemma3.1} in order to obtain
\begin{align*}
\|g_{\lambda ,t_0}& \|_{L^2(Q_{\delta T/\sqrt{8}})}
\\
&\le \varpi (2T+1)(\delta T)^{-1}  e^{-\lambda (\delta T)^2/16}\left( \|f\|_{L^2(R_{\delta /2})}^2+\|\partial_t f\|_{L^2(R_{\delta/2)}}^2\right)^{1/2}.
\end{align*}
That is we have
\[
\|g_{\lambda ,t_0} \|_{L^2(Q_{\delta T/\sqrt{8}})}\le \varpi (2T+1)(\delta T)^{-1}  e^{-\lambda (\delta T)^2/16}\|f\|_{H^1(J_{\delta /2},L^2(\Omega))}.
\]
This is the expected inequality.
\end{proof}

\begin{lemma}\label{lemma3.3}
Let $\lambda _0>0$. For any $f\in C^{j,\alpha}(\overline{Q}_T)$, $j=0,1$, and $\lambda \ge \lambda_0$, we have
\begin{equation}\label{e4}
\|\mathscr{F}_{\lambda ,t_0} f\|_{C^{j,\alpha}\left(\overline{Q}_{\delta T/\sqrt{8}}\right)}\le C\lambda \delta^{-j}e^{\lambda (\delta T)^2/16}\|f\|_{C^{j,\alpha}(\overline{Q}_T)},
\end{equation}
where
\[
C= 3(1+\varpi)\max \left[\frac{1}{\lambda_0},\left(\frac{T}{2}\right)^{1-\alpha}\right].
\]
\end{lemma}

\begin{proof}
Consider first the case $j=0$. Pick $(x_1,\tau _1), (x_2,\tau_2)\in \overline{Q}_{(\delta T)/\sqrt{8}}$. Then we have 
\begin{align*}
\mathscr{F}_{\lambda ,t_0} f(x_1,\tau_1)&-\mathscr{F}_{\lambda ,t_0} f(x_2,\tau_2)=\int_{\mathbb{R}}e^{-\lambda(i\tau_1 -(t-t_0))^2/2}\chi (t)[f(x_1,t)-f(x_2,t)]dt
\\
&\qquad +\int_{\mathbb{R}}\left[e^{-\lambda(i\tau_1 -(t-t_0))^2/2}-e^{-\lambda(i\tau_2 -(t-t_0))^2/2}\right]\chi (t)f(x_2,t)dt.
\end{align*}
Elementary computations show
\begin{align}
&\left|\int_{\mathbb{R}}e^{-\lambda(i\tau_1 -(t-t_0))^2/2}\chi (t)[f(x_1,t)-f(x_2,t)]dt\right| \label{2.2}
\\
&\hskip 3cm \le 2Te^{\lambda (\delta T)^2/16}|x_1-x_2|^\alpha\|f\|_{C^{0,\alpha}(\overline{Q}_T)}.\nonumber
\end{align}
On the other hand, since
\begin{align*}
&e^{-\lambda(i\tau_1 -(t-t_0))^2/2}-e^{-\lambda(i\tau_2 -(t-t_0))^2/2}
\\
&\quad =-i\lambda (\tau_1-\tau_2)\int_0^1e^{-\lambda(i(\tau_2+\rho (\tau_1-\tau_2)) -(t-t_0))^2/2}(i(\tau_2+\rho (\tau_1-\tau_2)) -(t-t_0))d\rho ,
\end{align*}
we have
\begin{equation}\label{2.3}
\left|e^{-\lambda(i\tau_1 -(t-t_0))^2/2}-e^{-\lambda(i\tau_2 -(t-t_0))^2/2}\right|\le \frac{5}{2}\lambda |\tau_1-\tau_2|Te^{\lambda (\delta T)^2/16}.
\end{equation}

Using \eqref{2.2} and \eqref{2.3}, we get
\[
\left[\mathscr{F}_{\lambda ,t_0} f\right]_\alpha \le 3\max \left[\frac{1}{\lambda_0},\left(\frac{T}{2}\right)^{1-\alpha}\right]\lambda Te^{\lambda (\delta T)^2/16}\|f\|_{C^{0,\alpha}(\overline{Q}_T)}.
\]
But
\[
\|\mathscr{F}_{\lambda ,t_0} f\|_{L^\infty (Q_{\delta T/\sqrt{8}})} \le 2Te^{\lambda (\delta T)^2/16}\|f\|_{L^\infty (Q_T)} .
\]
Whence
\[
\|\mathscr{F}_{\lambda ,t_0} f\|_{C^{0,\alpha}\left(\overline{Q}_{\delta T/\sqrt{8}}\right)}\le 3\max \left[\frac{1}{\lambda_0},\left(\frac{T}{2}\right)^{1-\alpha}\right]\lambda Te^{\lambda (\delta T)^2/16}\|f\|_{C^{0,\alpha}(\overline{Q}_T)}.
\]

We proceed now to the proof of $j=1$. In light of the following formula, that we already established above,
\[
\partial _\tau [\mathscr{F}_{\lambda ,t_0} f](x,\tau )=i\int_{\mathbb{R}}e^{-\lambda(i\tau -(t-t_0))^2/2}[\chi '(t)f(x,t)+\chi (t) \partial_t f(x,t)]dt,
\]
we have 
\begin{align*}
\left[\partial _\tau \mathscr{F}_{\lambda ,t_0} f\right]_\alpha &\le 3\varpi\max \left[\frac{1}{\lambda_0},\left(\frac{T}{2}\right)^{1-\alpha}\right]\lambda \delta^{-1}e^{\lambda (\delta T)^2/16}\|f\|_{C^{0,\alpha}(\overline{Q}_T)}
\\
&\hskip 1cm +3\max \left[\frac{1}{\lambda_0},\left(\frac{T}{2}\right)^{1-\alpha}\right]T\lambda e^{\lambda (\delta T)^2/16}\|\partial_tf\|_{C^{0,\alpha}(\overline{Q}_T)}
\end{align*}
and
\[
\|\partial _\tau \mathscr{F}_{\lambda ,t_0} f\|_{L^\infty (Q_{T/\sqrt{8}})}\le 2Te^{\lambda (\delta T)^2/16}\left(\|\partial _tf\|_{L^\infty (Q_T)} +\varpi (\delta T)^{-1}\|f\|_{L^\infty (Q_T)}\right) .
\]
Similarly,  we have also
\[
\left[\partial_{x_j}\mathscr{F}_{\lambda ,t_0} f\right]_\alpha \le 3\max \left[\frac{1}{\lambda_0},\left(\frac{T}{2}\right)^{1-\alpha}\right]\lambda Te^{\lambda (\delta T)^2/16}\|\partial_{x_j}f\|_{C^{0,\alpha}(\overline{Q}_T)},\quad 1\le j\le n .
\]
and 
\[
\|\partial _{x_j}\mathscr{F}_{\lambda ,t_0} f\|_{L^\infty (Q_{T/\sqrt{8}})}\le 2Te^{\lambda (\delta T)^2/16}\|\partial _{x_j}f\|_{L^\infty (Q_T)} ,\quad 1\le j\le n.
\]

We finally get 
\[
\|\mathscr{F}_{\lambda ,t_0} f\|_{C^{1,\alpha}\left(\overline{Q}_{\delta T/\sqrt{8}}\right)}\le C\lambda \delta ^{-1}e^{\lambda (\delta T)^2/16}\|f\|_{C^{1,\alpha}(\overline{Q}_T)}.
\]
The proof is then complete.
\end{proof}

\begin{lemma}\label{lemma3.4}
Fix $\lambda_0>0$ and $T_0>0$. Then there exists a constant $C>0$, only depending on $\lambda_0$ and $T_0$, so that, for any $\lambda \ge \lambda_0$, $T\ge T_0$, $f\in H^1(I_T,H^1(\Omega ))$ and $t_0\in I_{(1-\delta)T}$, we have
\begin{align}
C\|f(\cdot ,t_0)\|_{H^1(\Omega )}\le \lambda^{3/4}\|\mathscr{F}_{\lambda ,t_0}(f)&\|_{L^2(I_{\delta T/\sqrt{8}},H^1(\Omega ))} \label{e5}
\\
&+\delta^{-3/2}\lambda^{-1/4}\|f\|_{H^1(I_T,H^1(\Omega ))}.\nonumber
\end{align}
\end{lemma}

\begin{proof}
By density it is enough to give the proof for an arbitrary $f\in C^\infty (\overline{Q}_T)$. Pick then $f\in C^\infty (\overline{Q}_T)$ and let $h_\lambda =\mathscr{F}_{\lambda ,t_0} f$. Then, where $\mathcal{F}$ denotes the Fourier transform,
\begin{align*}
e^{-\tau ^2/(2\lambda)+i\tau t_0}h_\lambda \left(\cdot\, , \tau/\lambda\right)&=\int_{\mathbb{R}}e^{i\tau t}\left[ e^{-\lambda(t-t_0)^2/2}\chi (t) f(\cdot ,t) \right]dt
\\
&=2\pi \mathcal{F}^{-1}\left[ e^{-\lambda(t-t_0)^2/2}\chi (t) f(\cdot ,t) \right](\tau ).
\end{align*}
Therefore, where $x\in \overline{\Omega}$ is arbitrary,
\begin{equation}\label{E3.0}
e^{-\lambda(t-t_0)^2/2}\chi (t) f(x,t)=\frac{1}{2\pi}\int_{\mathbb{R}}e^{-i\tau t}e^{-\tau ^2/(2\lambda)+i\tau t_0}h_\lambda \left(x , \tau/\lambda\right) d\tau .
\end{equation}
We find, by applying Plancherel-Parseval's inequality and then making the change of variable $\tau =\lambda \rho$,
\begin{align}
\|e^{-\lambda(t-t_0)^2/2}\chi f(x,\cdot \,)\|_{L^2(\mathbb{R})}&=\frac{1}{\sqrt{2\pi}}\left\|e^{-\tau ^2/(2\lambda)}h_\lambda \left(x , \cdot /\lambda\right)\right\|_{L^2(\mathbb{R})}\label{E3.2}
\\
&=\frac{\sqrt{\lambda}}{\sqrt{2\pi}}\left\|e^{-\lambda \tau ^2/2}h_\lambda (x , \cdot )\right\|_{L^2(\mathbb{R})}. \nonumber
\end{align}
In particular
\begin{equation}\label{E3.3}
\left\|e^{-\lambda \tau ^2/2}h_\lambda (x , \cdot )\right\|_{L^2(\mathbb{R})}\le \frac{\sqrt{2\pi}}{\sqrt{\lambda}}\|f(x,\cdot \,)\|_{L^2(I_T)},
\end{equation}
from which we deduce that
\begin{equation}\label{E3.4}
\left\|e^{-\lambda \tau ^2/2}\mathscr{F}_{\lambda,t_0}(f)\right\|_{L^2(\Omega \times \mathbb{R})}\le \frac{\sqrt{2\pi}}{\sqrt{\lambda}}\|f\|_{L^2(Q_T)}.
\end{equation}

Again, the change of variable $\tau =\lambda \rho$ in \eqref{E3.0} yields
\begin{equation}\label{E3.1}
e^{-\lambda(t-t_0)^2/2}\chi (t) f(x,t)=\frac{\lambda}{2\pi}\int_{\mathbb{R}}e^{-i\lambda\tau t}e^{-\lambda \tau ^2/2+i\lambda \tau t_0}h_\lambda (x , \tau) d\tau .
\end{equation}
Hence, we get by taking $t=t_0$ in \eqref{E3.1},
\[
f(x,t_0)=\frac{\lambda}{2\pi}\int_{\mathbb{R}}e^{-\lambda \tau ^2/2}h_\lambda (x , \tau) d\tau .
\]
We decompose $f(x,t_0)$ into two terms: $f(x,t_0)=f_1(x,t_0)+f_2(x,t_0)$ with
\begin{align*}
&f_1(x,t_0)=\frac{\lambda}{2\pi}\int_{|\tau|\le \delta T/\sqrt{8}}e^{-\lambda \tau ^2/2}h_\lambda (x , \tau) d\tau ,
\\
&f_2(x,t_0)=\frac{\lambda}{2\pi}\int_{|\tau|>  \delta T/\sqrt{8}}e^{-\lambda \tau ^2/2}h_\lambda (x , \tau) d\tau .
\end{align*}
We obtain by applying again Cauchy-Schwaz's inequality 
\begin{align*}
|f_1(x,t_0)|^2&\le \left(\frac{\lambda^2}{4\pi^2}\int_{|\tau|\le \delta T/\sqrt{8}}e^{-\lambda \tau ^2}d\tau\right) \left(\int_{|\tau|\le \delta T/\sqrt{8}}|h_\lambda (x , \tau)|^2 d\tau\right)
\\
&\le \frac{\lambda^2}{4\pi^2}\frac{\sqrt{\pi}}{\sqrt{\lambda}}\|h(x,\cdot )\|_{L^2(I_{\delta T/\sqrt{8}})}^2,
\end{align*}
where we used that
\[
\int_{|\tau|\le \delta T/\sqrt{8}}e^{-\lambda \tau ^2}d\tau\le \int_{\mathbb{R}}e^{-\lambda \tau ^2}d\tau=\frac{1}{\sqrt{\lambda}} \int_{\mathbb{R}}e^{-s ^2}ds=\frac{\sqrt{\pi}}{\sqrt{\lambda}}.
\]

Therefore
\begin{equation}\label{E3.5}
\|f_1(\cdot ,t_0)\|_{L^2(\Omega )}\le \frac{\lambda^{3/4}}{2\pi^{3/4}}\|h_\lambda \|_{L^2\left(Q_{\delta T/\sqrt{8}}\right)}.
\end{equation}
On the other hand, we have, again according to Cauchy-Schwaz's inequality,
\begin{align*}
|e^{-\lambda \tau ^2/2}h_\lambda (x , \tau)|^2&\le \left(\int_{\mathbb{R}}e^{-\lambda (t-t_0)^2/2}\chi (t) |f(x,t)| dt\right)^2
\\
&\le \left(\int_{\mathbb{R}}e^{-\lambda (t-t_0)^2}\chi (t)dt\right)\left(\int_{\mathbb{R}}\chi (t) |f(x,t)|^2 dt\right)
\\
&\le \frac{1}{\sqrt{\lambda}}\left(\int_{\mathbb{R}}e^{-s^2}\chi (s/\sqrt{\lambda}+t_0)ds\right)\left(\int_{\mathbb{R}}\chi (t) |f(x,t)|^2 dt\right)
\\
&\le \frac{1}{\sqrt{\lambda}}\left(\int_{\mathbb{R}}e^{-s^2}ds\right)\left(\int_{\mathbb{R}}\chi (t) |f(x,t)|^2 dt\right)
\\
&\le \frac{\sqrt{\pi}}{\sqrt{\lambda}}\int_{\mathbb{R}}\chi (t) |f(x,t)|^2 dt.
\end{align*}
Whence
\begin{equation}\label{E3.6}
|e^{-\lambda \tau ^2/2}h_\lambda (x , \tau)|\le \frac{\pi^{1/4}}{\lambda^{1/4}}\|f(x,\cdot )\|_{L^2(I_T)}.
\end{equation}

In light of \eqref{E3.6}, we obtain by making an integration by parts
\begin{align*}
f_2(x,t_0)&=-\frac{\lambda}{2\pi}\int_{|\tau|> (\delta T)/\sqrt{8}}\frac{1}{\lambda \tau} \partial_\tau (e^{-\lambda \tau ^2/2})h_\lambda (x , \tau) d\tau 
\\
&=\frac{\sqrt{2}(\delta T)^{-1}}{\pi}e^{-\lambda (\delta T) ^2/16}\left(h_\lambda (x , -\delta T/\sqrt{8})+h_\lambda (x , \delta T/\sqrt{8})  \right)
\\
&\hskip 1cm -\frac{1}{2\pi}\int_{|\tau|> (\delta T)/\sqrt{8}}\frac{1}{\tau^2} e^{-\lambda \tau ^2/2}h_\lambda (x , \tau) d\tau
\\
&\hskip 2cm +\frac{1}{2\pi}\int_{|\tau|> (\delta T)/\sqrt{8}}\frac{1}{\tau} e^{-\lambda \tau ^2/2}\partial_\tau h_\lambda (x , \tau) d\tau .
\end{align*}
We write $f_2(x,t_0)=g_1(x,t_0)+g_2(x,t_0)+g_3(x,t_0)$, where
\begin{align*}
&g_1(x,t_0)=\frac{\sqrt{2}(\delta T)^{-1}}{\pi}e^{-\lambda (\delta T) ^2/16}\left(h_\lambda (x , -\delta T/\sqrt{8})+h_\lambda (x , \delta T/\sqrt{8})  \right),
\\
&g_2(x,t_0)=-\frac{1}{2\pi}\int_{|\tau|> (\delta T)/\sqrt{8}}\frac{1}{\tau^2} e^{-\lambda \tau ^2/2}h_\lambda (x , \tau) d\tau ,
\\
&g_3(x,t_0)=\frac{1}{2\pi}\int_{|\tau|> (\delta T)/\sqrt{8}}\frac{1}{\tau} e^{-\lambda \tau ^2/2}\partial_\tau h_\lambda (x , \tau) d\tau .
\end{align*}
It follows from \eqref{E3.6}
\begin{align*}
&|g_1(x,t_0)|\le \frac{2\sqrt{2}}{\pi^{3/4}\lambda^{1/4}}(\delta T)^{-1}\|f(x,\cdot )\|_{L^2(I_T)},
\\
&|g_2(x,t_0)|\le \frac{\sqrt{2}}{\pi^{3/4}\lambda^{1/4}}(\delta T)^{-1}\|f(x,\cdot )\|_{L^2(I_T)}.
\end{align*}
Hence
\begin{align}
&\|g_1(\cdot,t_0)\|\le \frac{2\sqrt{2}}{\pi^{3/4}\lambda^{1/4}}(\delta T)^{-1}\|f\|_{L^2(Q_T)},\label{E3.7}
\\
&\|g_2(\cdot,t_0)\|\le \frac{\sqrt{2}}{\pi^{3/4}\lambda^{1/4}}(\delta T)^{-1}\|f\|_{L^2(Q_T)}.\label{E3.8}
\end{align}

We have from the proof of Lemma \ref{lemma3.2}
\[
\partial_\tau h_\lambda (x,\tau )=i\mathscr{F}_{\lambda,t_0}(\partial_tf)(x,\tau)+k_\lambda(x,\tau ),
\]
with
\[
k_\lambda (x,\tau )=i\int_{\mathbb{R}}e^{-\lambda (i\tau-(t-t_0))^2/2}\chi'(t)f(x,t) dt.
\]
We find by proceeding similarly as for $h_\lambda$
\[
e^{-\lambda(t-t_0)^2/2}\chi '(t) f(x,t)=\frac{1}{2\pi}\int_{\mathbb{R}}e^{-i\tau t}e^{-\tau ^2/(2\lambda)+i\tau t_0}k_\lambda \left(x , \tau/\lambda\right) d\tau .
\]
Once again, Plancherel-Parseval's inequality and the change of variable $\tau =\lambda \rho$ yield
\begin{align*}
\|e^{-\lambda(t-t_0)^2/2}\chi ' f(x,\cdot \,)\|_{L^2(\mathbb{R})}&=\frac{1}{\sqrt{2\pi}}\left\|e^{-\tau ^2/(2\lambda)}k_\lambda \left(x , \cdot /\lambda\right)\right\|_{L^2(\mathbb{R})}
\\
&=\frac{\sqrt{\lambda}}{\sqrt{2\pi}}\left\|e^{-\lambda \tau ^2/2}k_\lambda (x , \cdot )\right\|_{L^2(\mathbb{R})}. 
\end{align*}
As $|\chi '|\le \varpi (\delta T)^{-1}$, we get in a straightforward manner 
\begin{equation}\label{E3.9}
\left\|e^{-\lambda \tau ^2/2}k_\lambda\right\|_{L^2(\Omega \times \mathbb{R})}\le \frac{\varpi\sqrt{2\pi}}{\sqrt{\lambda}}(\delta T)^{-1}\|f\|_{L^2(Q_T)}.
\end{equation}
In light of \eqref{E3.4}, with $f$ substituted by $\partial_tf$, and \eqref{E3.9} we obtain
\begin{equation}\label{E3.10}
\left\|e^{-\lambda \tau ^2/2}\partial_\tau h_\lambda\right\|_{L^2(\Omega \times \mathbb{R})}\le \frac{\sqrt{2\pi}}{\sqrt{\lambda}}\|\partial_tf\|_{L^2(Q_T)}+\frac{\varpi\sqrt{2\pi}}{\sqrt{\lambda}}(\delta T)^{-1}\|f\|_{L^2(Q_T)}.
\end{equation}
Cauchy-Schwarz's inequality then yields
\begin{align*}
|g_3(x,t_0)|^2&=\frac{1}{(2\pi)^2}\left(\int_{|\tau|> (\delta T)/\sqrt{8}}\frac{1}{\tau^2}d\tau\right) \left(\int_{|\tau|> (\delta T)/\sqrt{8}}e^{-\lambda \tau ^2}|\partial_\tau h_\lambda (x , \tau)|^2 d\tau \right)
\\
&\le \frac{\sqrt{8}}{(2\pi)^2}(\delta T)^{-1}\int_{\mathbb{R}}e^{-\lambda \tau ^2}|\partial_\tau h_\lambda (x , \tau)|^2 d\tau .
\end{align*}
This and \eqref{E3.10} give
\begin{equation}\label{E3.11}
\|g_3(\cdot ,t_0)\|_{L^2(\Omega )}\le \frac{2^{1/4}(\delta T)^{-1/2}}{\sqrt{\pi}}\frac{1}{\sqrt{\lambda}}\left( \|\partial_tf\|_{L^2(Q_T)}+(\delta T)^{-1}\|f\|_{L^2(Q_T)} \right).
\end{equation}
Let $\lambda_0$ be given. We see, by putting together \eqref{E3.5}, \eqref{E3.7}, \eqref{E3.8} and \eqref{E3.11}, that there exists a constant $C>0$, depending only on $\lambda_0$ and $T_0$, so that, for any $\lambda \ge \lambda_0$, we have
\begin{equation}\label{E3.12}
C\|f(\cdot ,t_0)\|_{L^2(\Omega )}\le \lambda^{3/4}\|\mathscr{F}_{\lambda ,t_0}(f)\|_{L^2(Q_{\delta T/\sqrt{8}})}+\delta^{-3/2}\lambda^{-1/4}\|f\|_{H^1(I_T,L^2(\Omega ))}.
\end{equation}
Noting that $\partial_i$ and $\mathscr{F}_{\lambda ,t_0}$ commute, we find by substituting in \eqref{E3.12} $f$ by $\partial_i f$, $1\le i\le n$,
\begin{align}
C\|\partial_if(\cdot ,t_0)\|_{L^2(\Omega )}\le \lambda^{3/4}\|\partial_i\mathscr{F}_{\lambda ,t_0}(f)&\|_{L^2(Q_{\delta T/\sqrt{8}})}\label{E3.13}
\\
&+\delta^{-3/2}\lambda^{-1/4}\|\partial_if\|_{H^1(I_T,L^2(\Omega ))}.\nonumber
\end{align}
The expected inequality follows readily from \eqref{E3.12} and \eqref{E3.13}. 
\end{proof}

\section{Stability of a Cauchy problem for wave equations}

We prove  Theorem \ref{maintheorem} in several steps. We recall that
\[
\mathcal{X}_T=C^{1,\alpha}(\overline{Q}_T)\cap H^1(I_T,H^1(\Omega ))
\]
 is endowed with its natural norm as an intersection of two Banach spaces, and $P$ is the anisotropic wave operator acting on $Q_T$ as follows 
\[
Pu=\mbox{div} (A\nabla u )-\partial _t^2u,
\]
where  $A=(a^{ij})$ is a symmetric matrix whose coefficients belong to $W^{1,\infty}(\Omega)$ and satisfy \eqref{main1} and \eqref{main2}.

For notational convenience, the gradient with respect to both variables $(x,t)$ and $(x,\tau)$ is denoted by $\nabla$. While the gradient with respect to $x$ is denoted by $\nabla'$.

Given $T\ge T_0>0$, we fix $0<\underline{\rho}<\overline{\rho}<T_0/\sqrt{8}$ and set 
\[
\underline{\delta}=\frac{\sqrt{8}\, \underline{\rho}}{T},\quad \overline{\delta}=\frac{\sqrt{8}\, \overline{\rho}}{T}, \quad \mathfrak{I}=\left[ \underline{\delta},\overline{\delta}\right].
\]
Of course $\underline{\delta}=\underline{\delta}(\underline{\rho},T)$ and $\overline{\delta}=\overline{\delta}(\overline{\rho},T)$.

\begin{proposition}\label{proposition3.1}
Let $T_0>0$. There exist two constant $C>0$ and $c>0$, only depending on $\Omega$, $\kappa$, $\varkappa$, $\underline{\rho}$, $\overline{\rho}$, $T_0$ and $\alpha$,  so that, for any $T\ge T_0$, $u\in C^{1,\alpha}(\overline{Q}_T)\cap H^2(Q_T)$, $\delta \in \mathfrak{I}$ and $t_0\in I_{(1-\delta)T}$, we have
\begin{align*}
&C\|u(\cdot ,t_0)\|_{H^1(\Omega )} \le \delta^2 \|u\|_{\mathcal{X}_T}
\\
&\qquad+ \mathfrak{e}_2(c\delta^{-14})\left(\|Pu\|_{L^2(Q_T)}+\|u\|_{H^{1,1}(\Sigma _T)}+\|\partial_\nu u\|_{L^2(\Sigma_T)}+\|u\|_{H^1\left(J_{\delta/2},L^2(\Omega)\right)}\right).
\end{align*}
\end{proposition}

\begin{proof}
Set $v_\lambda (x,\tau )=\mathscr{F}_{\lambda ,t_0} u(x,\tau)$. Then
\[ 
Lv_\lambda = \mathscr{F}_{\lambda ,t_0} (Pu)+w_\lambda ,
\]
where
\[
L= \mbox{div} (A(x)\nabla\cdot )+\partial_\tau ^2.
\]
and
\[
w_\lambda (x,\tau)=-\int_{\mathbb{R}}e^{-\lambda(i\tau -(t-t_0))^2/2}\left[2\chi' (t)\partial _tu(x,t)+\chi ''(t)u(x,t)\right]dt.
\]

Lemma \ref{lemma3.1} and Lemma \ref{lemma3.2} enable us to get
\begin{align*}
&\|Lv_\lambda \|_{L^2\left(Q_{\delta T/\sqrt{8}}\right)}+\|v_\lambda \|_{L^2\left(\Sigma _{\delta T/\sqrt{8}}\right)}+\|\nabla v_\lambda \|_{L^2\left(\Sigma_{\delta T/\sqrt{8}}\right)}\label{2.5}
\\
&\hskip 2cm \le \sqrt{2T}e^{\lambda (\delta T)^2/16}\left(\|Pu\|_{L^2(Q_T)}+\|u\|_{L^2(\Sigma_T)}+\|\nabla u\|_{L^2(\Sigma_T)}\right)\nonumber
\\
&\hskip 4cm+ 2\varpi(2T+1)(\delta T)^{-1}e^{-\lambda (\delta T)^2/16}\|u\|_{H^1(J_{\delta/2},L^2(\Omega))}\nonumber
\end{align*}
and, according to \eqref{e4},
\begin{equation}\label{2.6}
\|v_\lambda \|_{C^{1,\alpha}\left(\overline{Q}_{\delta T/\sqrt{8}}\right)}\le C_0\lambda \delta^{-1}e^{\lambda (\delta T)^2/16}\|u\|_{C^{1,\alpha}(\overline{Q}_T)},
\end{equation}
where
\[
C_0=3(1+\varpi)\max \left[\frac{1}{\lambda_0},\left(\frac{T}{2}\right)^{1-\alpha}\right].
\]
We get by applying Theorem \ref{theorem-el2} 
\begin{align}
C\|v_\lambda &\|_{H^1\left(Q_{\delta T/\sqrt{8}}\right)} \label{2.7}
\\
& \le T^{1/2}e^{c/\epsilon}e^{\lambda (\delta T)^2/16}\left(\|Pu\|_{L^2(Q_T)}+\|u\|_{L^2(\Sigma_T)}+\|\nabla u\|_{L^2(\Sigma_T)}\right) \nonumber
\\
&\qquad +Te^{c/\epsilon}\delta^{-1}e^{-\lambda (\delta T)^2/16}\|u\|_{H^1\left(J_{\delta/2},L^2(\Omega)\right)}\nonumber
\\
&\hskip 4cm+T^{1-\alpha}\epsilon^{\beta} \delta^{-1}\lambda e^{\lambda (\delta T)^2/16}\|u\|_{C^{1,\alpha}(\overline{Q}_T)}.\nonumber
\end{align}
Here and until the end of this proof $C>0$ is a generic constant only  depending on $\Omega$, $T_0$, $\kappa$, $\varkappa$, $\underline{\rho}$, $\overline{\rho}$ and $\alpha$.

On the other hand, we have from \eqref{e5}, for $\lambda \ge \lambda_0=16/T^2$,
\begin{equation}\label{2.8}
C\|u(\cdot ,t_0)\|_{H^1(\Omega )}\le \lambda ^{3/4}\|v_\lambda \|_{L^2\left(I_{\delta T/\sqrt{8}},H^1(\Omega )\right)}+
\delta^{-3/2}\lambda ^{-1/4}\|u\|_{\mathcal{X}_T}.
\end{equation}
Then a combination of \eqref{2.7} and \eqref{2.8} gives
\begin{align*}
&C\|u(\cdot ,t_0)\|_{H^1(\Omega )}
\\
&\hskip .5cm\le T^{1/2}\lambda ^{3/4}e^{c/\epsilon}e^{\lambda (\delta T)^2/16}\left(\|Pu\|_{L^2(Q_T)}+\|u\|_{L^2(\Sigma_T)}+\|\nabla u\|_{L^2(\Sigma_T)}\right)
\\
&\hskip 1cm+T\lambda ^{3/4}e^{c/\epsilon}\delta^{-1}e^{-\lambda (\delta T)^2/16}\|u\|_{H^1\left(J_{\delta/2},L^2(\Omega)\right)} 
\\
&\hskip 1.5 cm+\left(T^{1-\alpha}\lambda ^{7/4}\delta^{-1}\epsilon^{\beta} e^{\lambda (\delta T)^2/16}+ \delta^{-3/2}\lambda ^{-1/4}\right)\|u\|_{\mathcal{X}_T},
\end{align*}

We take in this inequality $\lambda =16\delta^{-16} /T^2$ $(\ge16/T^2)$. We get by using that $\sqrt{8}\, \underline{\rho}\le \delta T\le \sqrt{8}\, \overline{\rho}$
\begin{align*}
&C\|u(\cdot ,t_0)\|_{H^1(\Omega )}
\\
&\hskip .5cm\le e^{c/\epsilon}e^{c_0\delta^{-14}}\left(\|Pu\|_{L^2(Q_T)}+\|u\|_{L^2(\Sigma_T)}+\|\nabla u\|_{L^2(\Sigma_T)}+
\|u\|_{H^1\left(J_{\delta/2},L^2(\Omega)\right)} \right)
\\
&\hskip 1.5 cm+\left(\epsilon^{\beta} e^{c_0\delta^{-14}}+ \delta^2\right)\|u\|_{\mathcal{X}_T},
\end{align*}
for some positive constant $c_0$. This inequality with $\epsilon$ chosen so that $\epsilon ^{\beta}=e^{-2c_0\delta^{-14}}$ yields
\begin{align*}
&C\|u(\cdot ,t_0)\|_{H^1(\Omega )} \le \delta^2 \|u\|_{\mathcal{X}_T}
\\
&\quad+ \mathfrak{e}_2(c\delta^{-14})\left(\|Pu\|_{L^2(Q_T)}+\|u\|_{H^{1,1}(\Sigma _T)}+\|\partial_\nu u\|_{L^2(\Sigma_T)}+\|u\|_{H^1\left(J_{\delta/2},L^2(\Omega)\right)},\right)
\end{align*}
as expected.
\end{proof}

The following lemma is a consequence of \cite[Lemma 3.1]{CY} with $(t_0,t_1)=(0,1)$ and a change of variable.

\begin{lemma}\label{lemmaH}
Let $X$ be a Banach space for the norm  $\|\cdot  \|$ and $s\in (0,1/2)$. There exists a constant $c>0$ so that, for any $t_0<t_1$ and $u\in H^s((t_0,t_1),X)$, we have
\[
\left\| \frac{u}{d^s}\right\|_{L^2((t_0,t_1),X)}\le c(t_1-t_0)^{-2s}\max(1,(t_1-t_0)^{-1+2s})\|u\|_{H^s((t_0,t_1),X)}.
\]
Here $d=d(t)=\min\{ |t-t_0|,|t-t_1|\}$.
\end{lemma}

\begin{corollary}\label{corollary3.1}
Let $s\in (0,1/2)$ and $T_0>0$. There exist two constants $C>0$ and $c>0$, only depending on $\Omega$, $\kappa$, $\varkappa$, $s$, $\underline{\rho}$, $\overline{\rho}$, $T_0$ and $\alpha$,  so that, for any $T\ge T_0$, $u\in C^{1,\alpha}(\overline{Q}_T)\cap H^2(Q_T)$ and $\delta \in \mathfrak{I}$, we have
\begin{align*}
&C\|u\|_{L^2(I_T,H^1(\Omega ))} \le \delta ^s \|u\|_{\mathcal{X}_T}
\\
&\hskip 2cm+ \mathfrak{e}_2(c\delta^{-14})\Big(\|Pu\|_{L^2(Q_T)}+\|u\|_{H^{1,1}(\Sigma _T)} 
\\
&\hskip 4.5cm  +\|\partial_\nu u\|_{L^2(\Sigma_T)}+\|u\|_{H^1\left(J_{\overline{\delta}/2},L^2(\Omega)\right)}\Big).
\end{align*} 
\end{corollary}

\begin{proof}
First, noting that $J_{\delta/2}\subset J_{\overline{\delta}/2}$, $0<\delta \le \overline{\delta}$, we get by integrating, with respect to $t_0$, both sides of the inequality in Proposition \ref{proposition3.1}
\begin{align}
&C\|u\|_{L^2(I_{(1-\delta)T},H^1(\Omega ))} \le \delta \|u\|_{\mathcal{X}_\alpha (Q_T)} \label{e1+}
\\
&\hskip 2cm+ \mathfrak{e}_2(c\delta^{-14})\Big(\|Pu\|_{L^2(Q_T)}+\|u\|_{H^{1,1}(\Sigma _T)} \nonumber
\\
&\hskip 4.5cm +\|\partial_\nu u\|_{L^2(\Sigma_T)}+\|u\|_{H^1\left(J_{\overline{\delta}/2},L^2(\Omega).\right)}\Big),\nonumber
\end{align}
where we used $\left|I_{(1-\delta)T}\right|\le 2T\le 2\sqrt{8}\, \overline{\rho}\delta^{-1}$.

On the other hand, we have from Lemma \ref{lemmaH}, by noting that 
\[
\sqrt{8}\, \underline{\rho}\le \delta T=T-(1-\delta)T\le \sqrt{8}\, \overline{\delta},
\]
\begin{equation}\label{e2+}
\|u\|_{L^2\left(J_{(1-\delta)T},H^1(\Omega )\right)}\le C'\delta ^s\|u\|_{H^1(I_T,H^1(\Omega ))}.
\end{equation}

The expected inequality follows in a straightforward manner by adding side by side inequalities \eqref{e1+} and \eqref{e2+}.
\end{proof}

The next step consists in showing that  the term $\|u\|_{H^1\left(J_{\overline{\delta}/2},L^2(\Omega)\right)}$ in the  inequality of Corollary \ref{corollary3.1} can be bounded by a quantity involving only boundary terms and $Pu$.

\begin{lemma}\label{lemma3.5}
For $u\in H^2(Q_T)$, we have
\begin{align*}
&e^{-c\delta^{-1}}\|u\|_{H^1\left(J_{\overline{\delta}/2}, L^2(\Omega )\right)}\le \|Pu\|_{H^1\left(J_{\overline{\delta}/2}, L^2(\Omega )\right)}
\\
&\hskip 3cm+\sum_{\epsilon \in \{-,+\}}\sum_{k=0}^1\|\nabla \partial_t^ku(\cdot ,\epsilon T)\|_{L^2\left(\Omega \right)}
\\
&\hskip 5cm+\|u\|_{H^1\left( J_{\overline{\delta}/2},L^2(\partial \Omega)\right)}+\|\partial_\nu u\|_{H^1\left( J_{\overline{\delta}/2},L^2(\partial \Omega)\right)},
\end{align*}
the constant $c$ only depends on $\overline{\rho}$.
\end{lemma}

\begin{proof}
Let $u\in H^2(Q_T)$. We then get by mimicking the proof of the usual energy estimates
\begin{align*}
\int_\Omega A\nabla' u(x ,t)&\cdot \nabla' u(x,t)dx+\int_\Omega \left[\partial _tu(x,t)\right]^2dx
\\
&=\int_t^T\int_\Omega Pu(x,s)\partial _tu(x,s)dxds +\int_t^T\int_{\partial \Omega}u(x,t)\partial_\nu u(x,t)d\sigma(x)dt
\\
&+\int_\Omega A\nabla' u(\cdot ,T)\cdot \nabla' u(\cdot ,T)dx+\int_\Omega \left[\partial _tu(\cdot ,T)\right]^2dx.
\end{align*}
Since 
\begin{align*}
\left| \int_t^T\int_\Omega Pu(x,s)\partial _tu(x,s)dxds\right|&\le \frac{1}{2}\int_t^T\int_\Omega Pu(x,s)^2dxds
\\
&\hskip 1cm+\frac{1}{2}\int_t^T\int_\Omega [\partial _tu(x,s)]^2dxds
\end{align*}
and
\begin{align*}
\int_t^T\int_{\partial \Omega}u(x,s)\partial_\nu u(x,s)dS(x)ds &\le \frac{1}{2}\int_t^T\int_{\partial \Omega}u^2(x,t)d\sigma(x)ds
\\
&\hskip 1cm+\frac{1}{2}\int_t^T\int_{\partial \Omega}(\partial_\nu u(x,s))^2d\sigma(x)ds,
\end{align*}
we find, where $t\in \left((1-\overline{\delta}/2)T,T\right)$,
\begin{align}
\int_\Omega A\nabla' u(x ,t)\cdot \nabla' u(x,t)dx&+\int_\Omega \left[\partial _tu(x,t)\right]^2dx \le \label{2.10}
\\
&\qquad \Phi+\frac{1}{2}\int_t^T\int_\Omega [\partial _tu(x,s)]^2dxds,\nonumber
\end{align}
with
\begin{align*}
&\Phi=\frac{1}{2}\int_{(1-\overline{\delta}/2)T}^T\int_\Omega Pu(x,s)^2dxds+\int_\Omega A\nabla' u(\cdot ,T)\cdot \nabla' u(\cdot ,T)dx+\int_\Omega \left[\partial _tu(\cdot ,T)\right]^2dx
\\
&\hskip 1.5 cm+\frac{1}{2}\int_{(1-\overline{\delta}/2)T}^T\int_{\partial \Omega}u^2(x,t)d\sigma(x)ds+\frac{1}{2}\int_{(1-\overline{\delta}/2)T}^T\int_{\partial \Omega}(\partial_\nu u(x,s))^2d\sigma(x)ds.
\end{align*}
In particular \eqref{2.10} yields
\[
\int_\Omega \left[\partial _tu(x,t)\right]^2dx \le \Phi+\frac{1}{2}\int_t^T\int_\Omega [\partial _tu(x,s)]^2dxds.
\]
We obtain by applying Gr\"onwall's lemma
\[
\int_\Omega \left[\partial _tu(x,t)\right]^2dx\le \Phi\int_t^T e^{(\tau-t)/2}d\tau \le Te^{T/2}\Phi \le e^{c\delta^{-1}}\Phi.
\]
Here and henceforward $c$ is a constant only depending on $\overline{\rho}$. 

This in \eqref{2.10} gives
\[
\int_\Omega A\nabla' u(x ,t)\cdot \nabla' u(x,t)dx\le e^{c\delta^{-1}}\Phi .
\]
As $A\nabla' u(x ,t)\cdot \nabla' u(x,t)\ge \kappa |\nabla' u(x,t)|^2$, we have
\[
\| \nabla u\|_{L^2\left(\Omega \times \left((1-\overline{\delta}/2)T,T\right)\right)}^2 \le e^{c\delta^{-1}}\Phi .
\]

Using that 
\[
w\in H^1(\Omega )\rightarrow \left( \|\nabla' u\|_{L^2(\Omega )}^2+\| u\|_{L^2(\partial \Omega )}^2\right)^{\frac{1}{2}}
\]
defines an equivalent norm on $H^1(\Omega )$, we deduce
\[
\|u\|_{L^2\left(\Omega \times \left((1-\overline{\delta}/2)T,T\right)\right)}^2\le e^{c\delta^{-1}} \Phi .
\]
That is
\begin{align*}
&e^{-c\delta^{-1}}\|u\|_{L^2\left(\Omega \times \left((1-\overline{\delta}/2)T,T\right)\right)}\le \|Pu\|_{L^2\left(\Omega \times \left((1-\overline{\delta}/2)T,T\right)\right)}+\|\nabla u(\cdot ,T)\|_{L^2\left(\Omega \right)}
\\
&\hskip 3.5cm +\|u\|_{L^2\left(\partial \Omega \times \left((1-\overline{\delta}/2)T,T\right)\right)}+\|\partial_\nu u\|_{L^2\left(\partial \Omega \times \left((1-\overline{\delta}/2)T,T\right)\right)}.
\end{align*}
Since we have the same inequality when $u$ is substituted by $\partial_tu$, we get
\begin{align}
e^{-c\delta^{-1}}&\|u\|_{H^1\left(\left((1-\overline{\delta}/2)T,T\right), L^2(\Omega )\right)}\le \|Pu\|_{H^1\left(\left((1-\overline{\delta}/2)T,T\right), L^2(\Omega )\right)}\label{2.11}
\\
&\hskip.5cm+\|\nabla u(\cdot ,T)\|_{L^2\left(\Omega \right)} +\|\nabla \partial _tu(\cdot ,T)\|_{L^2\left(\Omega \right)}\nonumber
\\ 
&\hskip 1cm+\|u\|_{H^1\left( \left((1-\overline{\delta}/2)T,T\right),L^2(\partial \Omega)\right)}+\|\partial_\nu u\|_{H^1\left( \left((1-\overline{\delta}/2)T,T\right),L^2(\partial \Omega)\right)}.\nonumber
\end{align}
We obtain similarly
\begin{align}
&e^{-c\delta^{-1}}\|u\|_{H^1\left(\left(-T,-(1-\overline{\delta}/2)T\right), L^2(\Omega )\right)}\label{2.12}
\\
&\le \|Pu\|_{H^1\left(\left(-T,-(1-\overline{\delta}/2)T\right), L^2(\Omega )\right)} \nonumber
\\
&\quad +\|\nabla u(\cdot ,-T)\|_{L^2\left(\Omega \right)}+\|\nabla \partial _tu(\cdot ,-T)\|_{L^2\left(\Omega \right)}\nonumber
\\
&\qquad+\|u\|_{H^1\left( \left(-T,-(1-\overline{\delta}/2)T\right),L^2(\partial \Omega)\right)}+\|\nabla u\|_{H^1\left( \left(-T,-(1-\overline{\delta}/2)T\right),L^2(\partial \Omega)\right)}.\nonumber
\end{align}
The expected inequality follows by putting together \eqref{2.11} and \eqref{2.12}.
\end{proof}

Consider the following notation
\begin{align*}
&\mathcal{D}_{\overline{\delta}}(u)=\|Pu\|_{H^1\left(I_T, L^2(\Omega )\right)} +\|u\|_{H^{1,1}(\Sigma _T)}+\|\partial_\nu u\|_{L^2(\Sigma_T)}
\\
&\quad +\sum_{\epsilon\in \{-,+\}}\sum_{k=0}^1\|\nabla \partial _t^ku(\cdot ,\epsilon T)\|_{L^2(\Omega )}+\|u\|_{H^1\left(J_{\overline{\delta}/2},L^2(\partial \Omega)\right)}+\|\partial_\nu u\|_{H^1\left( J_{\overline{\delta}/2},L^2(\partial \Omega)\right)}.
\end{align*}

In light of Corollary \ref{corollary3.1} and Lemma 3.5, we can state the following result.
\begin{theorem}\label{theorem+3.1}
Let $s\in (0,1/2)$ and $T_0>0$. There exist two constants $C>0$ and $c>0$, only depending on $\Omega$, $T_0$, $\kappa$, $\varkappa$, $s$, $\underline{\rho}$, $\overline{\rho}$ and $\alpha$,  so that, for any $T\ge T_0$, $u\in C^{1,\alpha}(\overline{Q}_T)\cap H^2(Q_T)$ and $\delta \in \mathfrak{I}$, we have
\[
C\|u\|_{L^2\left(I_T, H^1(\Omega )\right)} \le \delta ^s\|u\|_{\mathcal{X}_T} + \mathfrak{e}_2(c\delta^{-14})\mathcal {D}_{\overline{\delta}}(u).
\] 
\end{theorem}

Now, in order to complete the proof of Theorem \ref{maintheorem}, we  prove that the data at $\pm T$ in $\mathcal{D}_{\overline{\delta}}(u)$ can be bounded by a quantity involving only lateral boundary terms. Prior to do that, we introduce an extension operator. Fix $a<b$, let $\varphi \in C_0^\infty ((a,b),H^{3/2}(\partial \Omega ))$ and denote by $E\varphi (\cdot ,t)\in H^2(\Omega )$ the unique solution of the BVP
\[
\left\{
\begin{array}{ll}
\Delta \psi =0\quad &\mbox{in}\; \Omega ,
\\
\psi =\varphi(\cdot ,t) &\mbox{on}\; \partial \Omega .
\end{array}
\right.
\]
By classical elliptic a priori estimates we have
\[
\|E\varphi (\cdot ,t)\|_{H^2(\Omega )}\le c_0\|\varphi (t)\|_{H^{3/2}(\partial \Omega )},
\]
the constant $c_0$ only depends on $\Omega$. 

One can check in a straightforward manner that $E\varphi \in C_0^\infty ((a,b),H^2(\Omega ))$ with $\partial _t^kE\varphi = E\partial_t^k\varphi$ and hence
\[
\|\partial _t^kE\varphi (\cdot ,t)\|_{H^2(\Omega )}\le c_0\|\partial _t^k\varphi (\cdot ,t)\|_{H^{3/2}(\partial \Omega )}.
\]
In consequence $E$ is extended as a bounded operator, still denoted by $E$, from $H^k((a,b),H^{3/2}(\partial \Omega ))$ into $H^k((a,b),H^2(\Omega ))$, $k\ge 0$ is an integer, in such a way that
\[
\|E\varphi\|_{H^k((a,b),H^2(\Omega ))}\le c_0\|\varphi \|_{H^k((a,b),H^{3/2}(\partial \Omega ))}.
\]
By Observing that $H^2((a,b),H^2(\Omega ))$ is continuously embedded in $H^2(\Omega \times (a,b))$, we have in particular
\begin{equation}\label{a1}
\|E\varphi\|_{H^2(\Omega \times (a,b))}\le c_0\|\varphi \|_{H^2((a,b),H^{3/2}(\partial \Omega ))}.
\end{equation}

Recall that  $\mathfrak{d}$ is defined by
\[
\mathfrak{d}=\sup_{y\in \Omega}\inf_{x\in \partial \Omega} d(x,y),
\]
with
\begin{align*}
d(x,y)=\inf\{L(\gamma ); \gamma :[0,1]&\rightarrow \overline{\Omega}
\\
& C^1\mbox{-piecewise path so that}\; \gamma (0)=x,\; \gamma (1)=y\},
\end{align*}
where $L(\gamma)$ denotes the length of $\gamma$.

\begin{proposition}\label{proposition3.2}
Let $\mathbf{c}_1>\mathbf{c}_0>\mathfrak{d}$. There exist three constants $C>0$, $c>0$ and $\mu_0>0$, only depending on $\Omega$, $\kappa$, $\varkappa$, $\mathbf{c}_0$ and $\mathbf{c}_1$, so that, for any $\mathbf{c}_0\le \mathbf{c}\le \mathbf{c}_1$, $b-a=\mathbf{c}$, $\mu \ge \mu_0$ and $u\in H^3((a,b),H^2(\Omega ))$, we have
\begin{align}
&\|\partial_tu(\cdot ,b)\|_{L^2(\Omega )}+\|\partial_t^2u(\cdot ,b)\|_{L^2(\Omega )}\label{a5.0}
\\ 
&\qquad \le Ce^{c\mu}\left( \|Pu\|_{H^1((a,b),L^2(\Omega ))}+\|u\|_{H^3((a,b),H^{3/2}(\partial \Omega ))}\right.\nonumber
\\ 
&\hskip 2cm \left. +\|\partial_\nu u\|_{H^1((a,b); L^2(\partial \Omega ))}\right)+\frac{C}{\mu^{1/2}} \|u\|_{H^3((a,b),H^1(\Omega ))}\nonumber
\end{align}
and
\begin{align}
&\|\nabla' u(\cdot ,b)\|_{L^2(\Omega )}+\|\nabla' \partial_tu(\cdot ,b)\|_{L^2(\Omega )}\label{a6.0}
\\
&\qquad \le Ce^{c\mu}\left( \|Pu\|_{L^2(\Omega \times (a,b))}+\|u\|_{H^2((a,b),H^{3/2}(\partial \Omega ))}\right.\nonumber
\\
&\hskip 2cm \left. +\|\partial_\nu u\|_{L^2(\partial \Omega \times (a,b)}\right)+\frac{C}{\mu^{1/4}} \|u\|_{H^2((a,b),H^2(\Omega ))} .\nonumber
\end{align}
\end{proposition}

\begin{proof}
Let $u\in H^3((a,b),H^2(\Omega ))$. Then set $\varphi = u_{|\partial \Omega \times (a,b)}$ and $v=u-E\varphi$. From \cite[Theorem 6.1]{LL} or \cite[theorem1]{Ro95}, there exist three constants $C>0$, $c>0$ and $\mu_0>0$, only depending on $\Omega$, $\kappa$, $\varkappa$ and $\mathbf{c}_0=b-a$, so that, for any $\mu \ge \mu_0$, we have
\begin{align*}
\|v(\cdot ,b)&\|_{L^2(\Omega )}+\|\partial_tv(\cdot ,b)\|_{H^{-1}(\Omega )}
\\
&\le Ce^{c\mu}\left( \|Pu\|_{L^2(\Omega \times (a,b))}+\|PE\varphi \|_{L^2(\Omega \times (a,b))}\right.
\\
&\hskip 2cm \left. +\|\partial _\nu v\|_{L^2(\partial \Omega \times (a,b)}\right) +\frac{C}{\mu} \left(\|v(\cdot ,b)\|_{H^1(\Omega )}+\|\partial_tv(\cdot ,b)\|_{L^2(\Omega )}\right).
\end{align*}

On the other hand,
\begin{align*}
&\left| \|u(\cdot ,b)\|_{L^2(\Omega )}-\|v(\cdot ,b)\|_{L^2(\Omega )}\right|\le \|E\varphi (\cdot ,b)\|_{L^2(\Omega )}\le c_0\|u\|_{H^1((a,b)H^{3/2}(\partial \Omega ))},
\\
&\left| \|u(\cdot ,b)\|_{H^1(\Omega )}-\|v(\cdot ,b)\|_{H^1(\Omega )}\right|\le c_0\|u\|_{H^1((a,b)H^{3/2}(\partial \Omega ))},
\\
&\left|\|\partial _tu(\cdot ,b)\|_{H^{-1}(\Omega )}-\|\partial _tv(\cdot ,b)\|_{H^{-1}(\Omega )}\right|\le c_0\|u\|_{H^2((a,b),H^{3/2}(\partial \Omega ))},
\\
&\left|\|\partial _tu(\cdot ,b)\|_{L^2(\Omega )}-\|\partial _tv(\cdot ,b)\|_{L^2(\Omega )}\right|\le c_0\|u\|_{H^2((a,b),H^{3/2}(\partial \Omega ))},
\\
&\|PE\varphi\|_{L^2(\Omega \times (a,b))}\le C\|E\varphi\|_{H^2(\Omega \times (a,b))}\le C'\|u\|_{H^2((a,b),H^{3/2}(\partial \Omega ))},
\\
&\|\partial _\nu v\|_{L^2(\partial \Omega \times (a,b))}\le \|\partial_\nu u\|_{L^2(\partial \Omega \times (a,b))}+c_0\|u\|_{L^2((a,b),H^{3/2}(\partial \Omega ))}.
\end{align*}
Whence
\begin{align}
&\|u(\cdot ,b)\|_{L^2(\Omega )}+\|\partial_tu(\cdot ,b)\|_{H^{-1}(\Omega )}\label{a2}
\\
&\qquad \le Ce^{c\mu}\left( \|Pu\|_{L^2(\Omega \times (a,b))}+\|u\|_{H^2((a,b),H^{3/2}(\partial \Omega ))}+\|\partial_\nu u\|_{L^2(\partial \Omega \times (a,b)}\right)\nonumber
\\
&\hskip 2cm +\frac{C}{\mu} \left(\|u(\cdot ,b)\|_{H^1(\Omega )}+\|\partial_tu(\cdot ,b)\|_{L^2(\Omega )}\right).\nonumber
\end{align}
But from the usual interpolation inequalities we have
\[
\|\partial_tu(\cdot ,b)\|_{L^2(\Omega )}\le c_1\|\partial_tu(\cdot ,b)\|_{H^{-1}(\Omega )}^{1/2}\|\partial_tu(\cdot ,b)\|_{H^{1}(\Omega )}^{1/2},
\]
the constant $c_1$ only depends on $\Omega$. Hence
\[
C\|\partial_tu(\cdot ,b)\|_{L^2(\Omega )}\le \epsilon \|\partial_tu(\cdot ,b)\|_{H^{1}(\Omega )}+\epsilon ^{-1}\|\partial_tu(\cdot ,b)\|_{H^{-1}(\Omega )},\;\; \epsilon >0.
\]
This in \eqref{a2} yields
\begin{align}
&\|u(\cdot ,b)\|_{L^2(\Omega )}+\|\partial_tu(\cdot ,b)\|_{L^2(\Omega )}\label{a3}
\\
&\quad\le Ce^{c\mu}\epsilon^{-1}\left( \|Pu\|_{L^2(\Omega \times (a,b))}+\|u\|_{H^2((a,b),H^{3/2}(\partial \Omega ))}+\|\partial_\nu u\|_{L^2(\partial \Omega \times (a,b)}\right)\nonumber
\\
&\hskip 1.5cm +\frac{C}{\mu}\epsilon^{-1} \left(\|u(\cdot ,b)\|_{H^1(\Omega )}+\|\partial_tu(\cdot ,b)\|_{L^2(\Omega )}\right)+\epsilon \|\partial_tu(\cdot ,b)\|_{H^1(\Omega )}.\nonumber
\end{align}
We get by taking $\epsilon=\mu ^{-1/2}$ in \eqref{a3} 
\begin{align*}
&\|u(\cdot ,b)\|_{L^2(\Omega )}+\|\partial_tu(\cdot ,b)\|_{L^2(\Omega )}
\\
&\qquad \le Ce^{c\mu}\left( \|Pu\|_{L^2(\Omega \times (a,b))}+\|u\|_{H^2((a,b),H^{3/2}(\partial \Omega ))}+\|\partial_\nu u\|_{L^2(\partial \Omega \times (a,b)}\right)
\\
&\hskip 2cm +\frac{C}{\mu^{1/2}} \left(\|u(\cdot ,b)\|_{H^1(\Omega )}+\|\partial_tu(\cdot ,b)\|_{H^1(\Omega )}\right).
\end{align*}
Thus
\begin{align}
&\|u(\cdot ,b)\|_{L^2(\Omega )}+\|\partial_tu(\cdot ,b)\|_{L^2(\Omega )}\label{a4}
\\
&\qquad \le Ce^{c\mu}\left( \|Pu\|_{L^2(\Omega \times (a,b))}+\|u\|_{H^2((a,b),H^{3/2}(\partial \Omega ))}\right.\nonumber
\\
&\hskip 2cm \left. +\|\partial_\nu u\|_{L^2(\partial \Omega \times (a,b)}\right)+\frac{C}{\mu^{1/2}} \|u\|_{H^2((a,b),H^1(\Omega ))}.\nonumber
\end{align}
Now \eqref{a4} together with \eqref{a4} with $u$ substituted by $\partial_tu$ imply
\begin{align*}
&\|\partial_tu(\cdot ,b)\|_{L^2(\Omega )}+\|\partial_t^2u(\cdot ,b)\|_{L^2(\Omega )}
\\ 
& \quad \le Ce^{c\mu}\left( \|Pu\|_{H^1((a,b),L^2(\Omega ))}+\|u\|_{H^3((a,b),H^{3/2}(\partial \Omega ))}\right.
\\
&\hskip 2cm \left. +\|\partial_\nu u\|_{H^1((a,b); L^2(\partial \Omega ))}\right) +\frac{C}{\mu^{\frac{1}{2}}} \|u\|_{H^3((a,b),H^1(\Omega ))}.
\end{align*}
That is we proved \eqref{a5.0} for $\mathbf{c_0}$.

We get  again from the usual interpolation inequalities
\[
C\|\nabla' u(\cdot ,b)\|_{L^2(\Omega )}\le \epsilon \|u(\cdot ,b)\|_{H^2(\Omega )}+\epsilon^{-1}\|u(\cdot ,b)\|_{L^2(\Omega )}.
\]
This in \eqref{a4} yields
\begin{align*}
&\|\nabla' u(\cdot ,b)\|_{L^2(\Omega )}+\|\nabla' \partial_tu(\cdot ,b)\|_{L^2(\Omega )}
\\
&\qquad \le Ce^{c\mu}\epsilon^{-1}\left( \|Pu\|_{L^2(\Omega \times (a,b))}+\|u\|_{H^2((a,b),H^{3/2}(\partial \Omega ))}+\|\partial_\nu u\|_{L^2(\partial \Omega \times (a,b)}\right)
\\
&\hskip 2cm +\frac{C}{\mu^{1/2}} \epsilon^{-1}\|u\|_{H^2((a,b),H^1(\Omega ))}+\epsilon \|u\|_{H^2((a,b),H^2(\Omega ))} 
\end{align*}
and hence
\begin{align*}
&\|\nabla' u(\cdot ,b)\|_{L^2(\Omega )}+\|\nabla' \partial_tu(\cdot ,b)\|_{L^2(\Omega )}
\\
&\quad\le Ce^{c\mu}\left( \|Pu\|_{L^2(\Omega \times (a,b))}+\|u\|_{H^2((a,b),H^{3/2}(\partial \Omega ))}\right.
\\
&\hskip 2cm \left. +\|\partial_\nu u\|_{L^2(\partial \Omega \times (a,b)}\right)+\frac{C}{\mu^{1/4}} \|u\|_{H^2((a,b),H^2(\Omega ))} .
\end{align*}
This is exactly inequality \eqref{a6.0} for $\mathbf{c}_0$.

Next, let $\mathbf{c}_0\le \mathbf{c}\le \mathbf{c}_1$ and $b-a=\mathbf{c}$. If $u\in H^3((a,b),H^2(\Omega ))$ we set $a_0=\mathbf{c}_0a/\mathbf{c}$, $b_0=\mathbf{c}_0b/\mathbf{c}$ and  $v(x,t)=u(x,\mathbf{c}t/\mathbf{c}_0)$, $(x,t)\in \Omega \times (a_0,b_0)$. In that case we have
\[
\frac{\mathbf{c}^2}{\mathbf{c}_0^2}A(x)v(x,t)-\partial_t^2v(x,t)=\frac{\mathbf{c}^2}{\mathbf{c}_0^2}\left( A(x)u(x,\mathbf{c}t/\mathbf{c}_0)-\partial_tu(x,\mathbf{c}t/\mathbf{c}_0) \right).
\]
Checking carefully the proof of \cite[Theorem 1]{Ro95} we see that \eqref{a5.0} and \eqref{a6.0} can be obtained uniformly with respect to $\mathbf{c}$ for the operator $(\mathbf{c}^2/\mathbf{c}_0^2)A(x)-\partial_t^2$ in the domain $\Omega \times (a_0,b_0)$. These observations allow us to complete the proof.
\end{proof}

We defined in the introduction the following notations
\begin{align*}
&\mathcal{N}(u)=\|u\|_{\mathcal{Y}_T},
\\
&\mathcal{D}(u)=\|Pu\|_{H^1(I_T,L^2(\Omega ))}+\|u\|_{H^{1,1}(\Sigma _T)}+\|\partial_\nu u\|_{L^2(\Sigma_T)},
\\
&\mathfrak{D}_{\overline{\delta}}(u)=\|u\|_{H^3\left(J_{\overline{\delta}/2},H^{3/2}(\partial \Omega )\right)}+\|\partial_\nu u\|_{H^1\left(J_{\overline{\delta}/2}, L^2(\partial \Omega )\right)},
\end{align*}
and set
\[
\tilde{\mathcal{D}}(u)=\sum_{\epsilon\in \{-,+\}}\sum_{k=0}^1\|\nabla \partial _t^ku(\cdot ,\epsilon T)\|_{L^2(\Omega )}.
\]

Let  $T_0>\mathfrak{d}$, $\mathfrak{d}<\sqrt{8}\, \overline{\rho}<T_0$ and $0<\underline{\rho}<\overline{\rho}$. In that case $\overline{\delta}T=\sqrt{8}\, \overline{\rho}>\mathfrak{d}$ for any $T\ge T_0$. Therefore, since $\sqrt{8}\, \underline{\rho}\le \delta T\le \sqrt{8}\, \overline{\rho}$, we deduce from \eqref{a5.0} and \eqref{a6.0}
\begin{equation}\label{a7}
C\tilde{\mathcal{D}}(u)\le e^{c\mu}\left(\mathcal{D}(u)+\mathfrak{D}_{\overline{\delta}}(u)\right)+\mu^{-1/4}\mathcal{N}(u),\quad \mu \ge \mu_0,
\end{equation}
the constants $C$, $c$ and $\mu_0$ only depend on $\Omega$, $T_0$, $\kappa$, $\varkappa$ and $\overline{\rho}$.

This and Theorem \ref{theorem+3.1} give, with $\delta \in \mathfrak{I}$,
\begin{align*}
&C\|u\|_{L^2\left(I_T,H^1(\Omega )\right)} \le \delta^s\mathcal{N}(u) 
\\
&\qquad+ \mathfrak{e}_2(c\delta^{-14})\left(\mathcal{D}(u)+e^{c\mu}(\mathcal{D}(u)+\mathfrak{D}_{\overline{\delta}}(u))+\mu^{-1/4}\mathcal{N}(u)+\mathfrak{D}_{\overline{\delta}}(u)\right),
\end{align*}
the constants $C$, $c$ and $\mu_0$ only depend on $\Omega$, $\kappa$, $\varkappa$, $T_0$, $\underline{\rho}$ and $\overline{\rho}$.

In this inequality, we may substitute $c$ by $\lambda c$ with $\lambda \ge 1$ sufficiently large in such a way that $\overline{\delta}^{-s}\mathfrak{e}_2(\lambda c\overline{\delta}^{-14})>\mu_0^{1/4}$.

We take $\mu\ge \mu_0$ in the preceding inequality so that $\mathfrak{e}_2(c\delta^{-14})\mu^{-\frac{1}{4}}=\delta^s$ in order to obtain 
\[
C\|u\|_{L^2\left(I_T,H^1(\Omega )\right)} \le \delta ^s\mathcal{N}(u) + \mathfrak{e}_3(c\delta^{-14})\left(\mathcal{D}(u)+\mathfrak{D}_{\overline{\delta}}(u) \right).
\]
In other words we proved Theorem \ref{maintheorem}.

A straightforward consequence of Theorem \ref{maintheorem} is the following uniqueness of continuation result.

\begin{corollary}\label{corollary1++}
Let $T_0>\mathfrak{d}$ and $\mathfrak{d}<\sqrt{8}\, \overline{\rho}<T_0$.  If $T\ge T_0$ and  $u\in \mathcal{Y}_T$ satisfies $Pu=0$ in $Q_T$ then $u=\partial_\nu u=0$ on $(J_{\overline{\delta}/2}\times \partial \Omega)\cup \Sigma_T$ implies $u=0$ in $Q_T$.
\end{corollary}

Similarly to the elliptic case, we deduce from Theorem \ref{maintheorem} a stability estimate. Let $\varrho_\ast =e^{-e}$ and define, for $\beta >0$ and $\varrho_0\le \varrho_\ast$,
\[
\Psi_{\varrho_0,\beta}(\varrho )=
\left\{
\begin{array}{ll}
(\ln \ln  |\ln \varrho |)^{-\beta } \quad &\mbox{if}\; 0<\varrho\le \varrho_0 ,
\\
\varrho &\mbox{if}\; \varrho \ge \varrho_0.
\end{array}
\right.
\]

\begin{corollary}\label{corollary2++}
Let $s\in (0,1/2)$, $\mathfrak{b}>0$ and assume that $T_0>\mathfrak{d}$, $\mathfrak{d}<\sqrt{8}\, \overline{\rho}<T_0$ and $0<\underline{\rho}<\overline{\rho}$. Then there exists a constant $\hat{T}\ge T_0$, only depending of $\Omega$, $\kappa$, $\varkappa$, $\underline{\rho}$, $\overline{\rho}$, $s$, $\alpha$ and $\mathfrak{b}$, so that, for any $T\ge \hat{T}$, $u\in\mathcal{Y}_T$ satisfying $u\ne 0$ and 
\[
\frac{\mathcal{D}(u)+\mathfrak{D}_{\overline{\delta}}(u)}{\mathcal{N}(u)}\le \mathfrak{b},
\]
we have
\[
C\|u\|_{L^2\left(I_T,H^1(\Omega )\right)}\le \mathcal{N}(u)\Psi_{\tilde{\rho},s/8}\left( \mathfrak{b} \right),
\] 
where $C>0$ and $0<\tilde{\varrho}\le  \varrho_\ast$ depend only of $\Omega$, $\kappa$, $\varkappa$, $\underline{\rho}$, $\overline{\rho}$, $s$, $\alpha$,  $\mathfrak{b}$ and $T$.
\end{corollary}

\begin{proof}
 Let $C>0$ and $c>0$ be the constants in Theorem \ref{maintheorem}. We recall that $C$ and $c$ only depend of  $\Omega$, $T_0$, $\kappa$, $\varkappa$, $\underline{\rho}$, $\overline{\rho}$, $\alpha$ and $s$.
 
 Define the function $\ell$ by $\ell (\delta)=\delta ^s e^{-\mathfrak{e}_3(c\delta^{-14})}$. Since $\ell(\underline{\delta})$ is a decreasing function of $T$ and $\ell(\underline{\delta})$ converges to $0$ as $T$ goes to $\infty$, we find $\hat{T}\ge T_0$ so that
 \[
 \ell(\underline{\delta})<\min (\mathfrak{b},\varrho_\ast),\quad T\ge \hat{T}.
 \]
 
 For fixed $T\ge \hat{T}$, if $\mathfrak{b}<\min (\ell (\overline{\delta}) , \varrho_\ast)=\tilde{\varrho}$ then there exists $\underline{\delta}\le\tilde{\delta}\le \overline{\delta}$ so that $\ell (\tilde{\delta})=\mathfrak{b}$. Therefore, where $\hat{c}>0$ is a constant only depending of $\Omega$, $\kappa$, $\varkappa$, $\underline{\rho}$, $\overline{\rho}$, $s$, $\alpha$,  $\mathfrak{b}$ and $T$,
\[
\frac{1}{\mathfrak{b}}\le \mathfrak{e}_3(\hat{c}\tilde{\delta}^{-14}).
\]
Or equivalently
\[
\tilde{\delta}\le [\hat{c}\ln \ln |\ln \frak{b}|]^{-1/14}.
\]

Pick $u\in \mathcal{Y}_T$, $u\ne 0$, set
\[
\mathfrak{a}=\frac{\|u\|_{L^2\left(I_T,H^1(\Omega )\right)}}{\mathcal{N}(u)}
\]
and assume that
\[
\frac{\mathcal{D}(u)+\mathfrak{D}_{\overline{\delta}}(u)}{\mathcal{N}(u)}\le \mathfrak{b}.
\]
In light of \eqref{1++} we get
\begin{equation}\label{2++}
C\mathfrak{a} \le \delta ^s+ \mathfrak{e}_3(c\delta^{-14})\mathfrak{b} ,\quad \delta \in \mathfrak{I} ,
\end{equation}

Then it follows readily by taking $\delta=\tilde{\delta}$ ($\in  \mathfrak{I}$) in \eqref{2++} that
\begin{equation}\label{3++}
\hat{C}\mathfrak{a}\le [\ln \ln |\ln \mathfrak{b}|]^{-s/14},
\end{equation}
the constant $\hat{C}>0$ only depends $\Omega$, $\kappa$, $\varkappa$, $\underline{\rho}$, $\overline{\rho}$, $s$, $\alpha$,  $\mathfrak{b}$ and $T$.

If $\mathfrak{b}\ge \tilde{\varrho}$ then obviously we have
\begin{equation}\label{4++}
\mathfrak{a}\le 1\le \frac{\mathfrak{b}}{\tilde{\varrho}}.
\end{equation}
The result follows then from \eqref{3++} and \eqref{4++}.
\end{proof}

\section{Comments}

As before $\Gamma$ is a nonempty open subset of $\partial \Omega$. Define then
\[
\mathfrak{d}_\Gamma =\sup_{y\in \Omega}\inf_{x\in \Gamma} d(x,y),
\]
with
\begin{align*}
d(x,y)=\inf\{L(\gamma ); \gamma :[0,1]&\rightarrow \overline{\Omega}
\\
& C^1\mbox{-piecewise path so that}\; \gamma (0)=x,\; \gamma (1)=y\},
\end{align*}
where $L(\gamma)$ denotes the length of $\gamma$.

Similarly to the proof Proposition \ref{proposition3.2}, we get, by applying \cite[Theorem 6.1]{LL} and without using an extension operator, the following result.

\begin{proposition}\label{proposition-com1}
Let $\mathbf{c}_1>\mathbf{c}_0>\mathfrak{d}_\Gamma$. There exist three constants $C>0$, $c>0$ and $\mu_0>0$, only depending on $\Omega$, $\Gamma$, $\kappa$, $\varkappa$, so that, for any $\mathbf{c}_0\le \mathbf{c}\le \mathbf{c}_1$, $b-a=\mathbf{c}$, $\mu \ge \mu_0$ and $u\in H^3((a,b),H^2(\Omega )\cap H_0^1(\Omega ))$, we have
\begin{align*}
\|\partial_tu(\cdot ,b)\|_{L^2(\Omega )}&+\|\partial_t^2u(\cdot ,b)\|_{L^2(\Omega )}
\\ 
& \le Ce^{c\mu}\left( \|Pu\|_{H^1((a,b),L^2(\Omega ))}+\|\partial_\nu u\|_{H^1((a,b); L^2(\Gamma ))}\right)
\\ 
&\hskip 5cm +\frac{C}{\mu^{1/2}} \|u\|_{H^3((a,b),H^1(\Omega ))}
\end{align*}
and
\begin{align*}
\|\nabla' u(\cdot ,b)\|_{L^2(\Omega )}&+\|\nabla' \partial_tu(\cdot ,b)\|_{L^2(\Omega )}
\\
&\le Ce^{c\mu}\left( \|Pu\|_{L^2(\Omega \times (a,b))}+\|\partial_\nu u\|_{L^2(\Gamma \times (a,b)}\right)
\\
&\hskip 4cm +\frac{C}{\mu^{1/4}} \|u\|_{H^2((a,b),H^2(\Omega ))} .
\end{align*}
\end{proposition}

The following notations were already used in the previous sections
\begin{align*}
&\mathcal{N}(u)=\|u\|_{\mathcal{Y}_T},
\\
&\mathcal{D}(u)=\|Pu\|_{H^1(I_T,L^2(\Omega ))}+\|u\|_{H^{1,1}(\Sigma _T)}+\|\partial_\nu u\|_{L^2(\Sigma_T)},
\end{align*}

In light of Proposition \ref{proposition-com1}, we get by mimicking the proof of Theorem \ref{maintheorem} the following result.
\begin{theorem}\label{theorem-com1}
Let $s\in (0,1/2)$, $\Gamma$ be a nonempty open subset of $\partial \Omega$, $T_0>\mathfrak{d}_\Gamma$, $\mathfrak{d}_\Gamma <\sqrt{8}\, \overline{\rho}<T_0$ and $0<\underline{\rho}<\overline{\rho}$. There exist two constants $C>0$ and $c>0$, only depending on $\Omega$, $T_0$, $\kappa$, $\varkappa$, $\underline{\rho}$, $\overline{\rho}$, $s$ and $\alpha$, so that, for any  $T\ge T_0$, $u\in \mathcal{Y}_T$ satisfying $u=\partial_\nu u=0$ on $J_{\overline{\delta}/2}\times \partial \Omega$ and $\delta \in [\sqrt{8}\, \underline{\rho}/T,\sqrt{8}\,\overline{\rho}/T]$, we have
\[
C\|u\|_{L^2\left(I_T,H^1(\Omega )\right)} \le \delta ^s\mathcal{N}(u) + \mathfrak{e}_3\left(c\delta^{-14}\right)\mathcal{D}(u) .
\]
\end{theorem}

This theorem has to be compared with the following one which is contained in \cite[Theorem 6.3]{LL}.
\begin{theorem}\label{theorem-com2}
Let $\Gamma$ be a nonempty open subset of $\partial \Omega$ and  $T>\mathfrak{d}_\Gamma$. There exist four constants $C>0$, $c>0$, $\mu_0>0$
 and $0<\tilde{T}<T$, only depending on $\Omega$, $T$, $\kappa$, $\varkappa$, so that, for any  $u\in H^1(I_T,H^2(\Omega )\cap H_0^1(\Omega ))$ and $\mu \ge \mu_0$, we have
\[
C\|u\|_{L^2(Q_{\tilde{T}})} \le Ce^{c\mu}\left(\|\partial_\nu u\|_{L^2(\Sigma_T)}+\|Pu\|_{L^2(Q_T)}\right)+\frac{C}{\mu}\|u\|_{H^1(Q_T)}.
\]
\end{theorem}

It is worth mentioning that we have a variant of Theorem \ref{theorem-com1} for functions $u\in \mathcal{Y}_T$ satisfying $u=\partial_\nu u=0$ on $J_{\overline{\delta}/2}\times (\partial \Omega\setminus \Gamma_0 )$, $\Gamma_0\Subset \Gamma$. This follows by substituting in the proof of Proposition \ref{proposition3.2} $v=u-E\varphi$ by $v=u-E\chi \varphi$ with $\chi \in C_0^\infty (\mathbb{R}^n)$ satisfying $\mbox{supp}\chi\cap \partial \Omega \subset \Gamma$ and $\chi =1$ on $\Gamma_0$.

\section{A particular case}

We aim in this section to improve the result of Theorem \ref{maintheorem}. We precisely show that, under a ``smallness'' condition on the coefficients of the matrix $A$, the term $\mathfrak{e}_3(c\delta^{-14})$ in the inequality of Theorem \ref{maintheorem} can be improved by $\mathfrak{e}_2(c\delta^{-14})$. In other words we have a better stability result.

The analysis we carry out in this section consists in an adaptation of known ideas used to establish observability inequalities via the multiplier method for the wave equation $\partial_t^2-\Delta$. We refer for instance to \cite[Subsection 3.1, page 35]{Ko}. In the variable coefficients case similar approach was used in \cite{Ya} with a geometric viewpoint, i.e  by considering the metric generated by $(a^{ij})$. The condition on $A$ is then interpreted in term of the corresponding sectional curvature.

We start by establishing an identity involving a usual multiplier. To this end, fix $x_0\in \overline{\Omega}$ arbitrary and let $m(x)=x-x_0$. 

Let $u\in H^2(\Omega \times (a,b))$ satisfying $u=0$ on $\partial \Omega \times (a,b)$. We find by making an integration by parts
\begin{align}
\int_\Omega \mbox{div}(A\nabla' u)(m\cdot \nabla' u)dx=&-\int_\Omega A\nabla' u \cdot \nabla' (m\cdot \nabla' u)dx\label{b1}
\\
&\qquad +\int_{\partial \Omega}(A\nabla' u\cdot \nu) (m\cdot \nabla' u)d\sigma(x).\nonumber
\end{align}
Use $\nabla' u=\partial_\nu u\, \nu$ in \eqref{b1} in order to get
\begin{align}
\int_\Omega \mbox{div}(A\nabla' u)(m\cdot \nabla u)dx=&-\int_\Omega A\nabla' u \cdot \nabla' (m\cdot \nabla' u)dx\label{b2}
\\
&\qquad +\int_{\partial \Omega}(\partial _\nu u)^2(A\nu\cdot \nu) (m\cdot \nu)d\sigma(x).\nonumber
\end{align}
We have
\[
\sum_{k=1}^n\partial_{x_i}(m_k\partial_{x_k}u)=\sum_{k=1}^n\left(\delta_{ik}\partial_{x_k}u+m_k \partial _{x_k}\partial_{x_i}u\right)=\partial_{x_i}u+m\cdot \nabla \partial_{x_i}u .
\]
Therefore
\[
A\nabla' u \cdot \nabla' (m\cdot \nabla' u)=A\nabla' u\cdot \nabla' u+A\nabla' u\cdot V,
\]
where $V=(v_1,\ldots, v_n)$ with $v_i=m\cdot \nabla' \partial_{x_i}u$.

But
\begin{align*}
&\int_{\Omega} A\nabla u\cdot Vdx=\sum_{i,j,k=1}^n \int_\Omega a^{ij} \partial_{x_j}um_k\partial_{x_k}\partial_{x_i}udx
\\
&\hskip.5cm =-\sum_{i,j,k=1}^n \int_\Omega a^{ij}\partial_{x_j}u\partial_{x_i}udx-\sum_{i,j,k=1}^n \int_\Omega a^{ij}\partial_{x_k}\partial_{x_j}um_k\partial_{x_i}udx
\\ 
&\hskip1cm -\sum_{i,j,k=1}^n \int_\Omega \partial_{x_k}a^{ij}\partial_{x_j}um_k\partial_{x_i}udx+\sum_{i,j,k=1}^n\int_{\partial \Omega}a^{ij}\partial_{x_j}um_k\partial_{x_i}u\nu_kd\sigma(x).
\end{align*}
This and the fact that $A$ is a symmetric matrix yield
\[
2\int_{\Omega} A\nabla' u\cdot Vdx=-n\int_\Omega A\nabla' u\cdot \nabla' udx-\int_\Omega B\nabla' u\cdot \nabla' udx+\int_{\partial \Omega} (\partial_\nu u)^2(A\nu \cdot \nu)(m\cdot \nu)d\sigma(x).
\]
Here $B=(b^{ij})$ with $b^{ij}=\nabla' a^{ij}\cdot m$.

Whence
\begin{align*}
\int_\Omega A\nabla' u \cdot \nabla' (m\cdot \nabla' u)dx= \left(1-\frac{n}{2}\right)\int_\Omega A\nabla' u&\cdot \nabla' udx-\frac{1}{2}\int_\Omega B\nabla' u\cdot \nabla' udx
\\
&+\frac{1}{2}\int_{\partial \Omega} (\partial_\nu u)^2(A\nu \cdot \nu)(m\cdot \nu)d\sigma(x).
\end{align*}
This in \eqref{b2} gives
\begin{align*}
\int_\Omega \mbox{div}(A\nabla' u)(m\cdot \nabla' u)dx=\left(\frac{n}{2}-1\right)\int_\Omega A\nabla' u&\cdot \nabla' udx+\frac{1}{2}\int_\Omega B\nabla' u\cdot \nabla' udx
\\
&+\frac{1}{2}\int_{\partial \Omega} (\partial_\nu u)^2(A\nu \cdot \nu)(m\cdot \nu)d\sigma(x)
\end{align*}
implying
\begin{align}
&\int_a^b\int_\Omega \mbox{div}(A\nabla' u)(m\cdot \nabla' u)dxdt=\left(\frac{n}{2}-1\right)\int_a^b\int_\Omega A\nabla' u\cdot \nabla udxdt \label{b3}
\\
&\hskip 1cm+\frac{1}{2}\int_a^b\int_\Omega B\nabla' u\cdot \nabla' udxdt+\int_a^b\frac{1}{2}\int_{\partial \Omega} (\partial_\nu u)^2(A\nu \cdot \nu)(m\cdot \nu)d\sigma(x)dt.\nonumber
\end{align}

On the other hand, we have
\begin{align*}
\int_\Omega \int_a^b\partial _t^2um\cdot \nabla' udtdx&=\int_\Omega \partial_tu(\cdot,b)m\cdot \nabla' u(\cdot ,b)dx
\\
&-\int_\Omega \partial_tu(\cdot,a)m\cdot \nabla' u(\cdot ,a)dx-\int_\Omega \int_a^b\partial _tum\cdot \nabla' \partial_tudtdx.
\end{align*}
An integration by parts gives
\[
\int_\Omega  \partial _tum\cdot \nabla' \partial_tudx=-\int_\Omega  m\cdot \nabla' \partial_t u \partial_t udx-n\int_\Omega (\partial_tu) ^2dx
\]
and then
\[
\int_\Omega  \partial _tu(m\cdot \nabla' \partial_tu)dx=-\frac{n}{2}\int_\Omega (\partial_tu) ^2dx.
\]
Thus 
\begin{align}
\int_\Omega \int_a^b&\partial _t^2u(m\cdot \nabla' u)dtdx=\int_\Omega \partial_tu(\cdot,b)(m\cdot \nabla' u(\cdot ,b))dx\label{b4}
\\
& -\int_\Omega \partial_tu(\cdot,a)(m\cdot \nabla' u(\cdot ,a))dx+\frac{n}{2}\int_a^b\int_\Omega (\partial_tu) ^2dxdt.\nonumber
\end{align}
We combine \eqref{b3} and \eqref{b4} in order to get
\begin{align}
2\int_a^b \int_\Omega &Pu(m\cdot \nabla' u)dxdt=\left(n-2\right)\int_a^b\int_\Omega A\nabla' u\cdot \nabla' udxdt \label{b5}
\\
&+\int_a^b\int_\Omega B\nabla' u\cdot \nabla' udxdt+\int_a^b\int_{\partial \Omega} (\partial_\nu u)^2(A\nu \cdot \nu)(m\cdot \nu)d\sigma(x)dt\nonumber
\\
&-2\int_\Omega \partial_tu(\cdot,b)(m\cdot \nabla' u(\cdot ,b))dx+2\int_\Omega \partial_tu(\cdot,a)(m\cdot \nabla' u(\cdot ,a))dx\nonumber
\\
&\hskip 5cm -n\int_a^b\int_\Omega (\partial_tu) ^2dxdt.\nonumber
\end{align}
We have also by simple integrations by parts
\begin{align}
\int_a^b \int_\Omega Puudxdt&=-\int_a^b\int_\Omega A\nabla' u\cdot \nabla' udxdt+\int_a^b\int_\Omega (\partial_tu) ^2dxdt \label{b6}
\\
&+\int_\Omega \partial_tu(\cdot,b)u(\cdot ,b)dx-\int_\Omega \partial_tu(\cdot,a)u(\cdot ,a)dx.\nonumber
\end{align}
Let
\[
w=2(m\cdot\nabla' u)+(n-1)u.
\]
Then, it follows from \eqref{b5} and \eqref{b6}
\begin{align}
\int_a^b \int_\Omega Puwdxdt&=-\int_a^b\int_\Omega A\nabla' u\cdot \nabla' udxdt \label{b7}
\\
&\hskip 1cm -\int_a^b\int_\Omega (\partial_tu) ^2dxdt+\int_a^b\int_\Omega B\nabla' u\cdot \nabla' udxdt\nonumber
\\
&\hskip 2cm -\int_\Omega \partial_tu(\cdot,b)w(\cdot ,b)dx+\int_\Omega \partial_tu(\cdot,a)w(\cdot ,a)dx\nonumber
\\
&\hskip 3cm +\int_a^b\int_{\partial \Omega} (\partial_\nu u)^2(A\nu \cdot \nu)(x\cdot \nu)d\sigma(x)dt.\nonumber
\end{align}

Define
\[
\mathcal{E}_u(t)=\int_\Omega \left[A\nabla' u(\cdot ,t)\cdot \nabla' u(\cdot, t)+(\partial_tu(\cdot ,t)) ^2\right]dx.
\]
We see that \eqref{b7} can be rewritten as follows
\begin{align}
\mathcal{E}_u=\int_a^b \mathcal{E}_u(t)dt\ &= \int_a^b\int_\Omega B\nabla' u\cdot \nabla' udxdt \label{b07}
\\
&\hskip 1cm -\left[\int_\Omega \partial_tu(\cdot,t)w(\cdot ,t)dx\right]_{t=a}^{t=b}-\int_a^b \int_\Omega Puwdxdt\nonumber
\\
&\hskip 2cm +\int_a^b\int_{\partial \Omega} (\partial_\nu u)^2(A\nu \cdot \nu)(m\cdot \nu)d\sigma(x)dt.\nonumber
\end{align}

In the sequel
\[
\mathfrak{d}_0=\max \{|x-y|;\; x,y\in \overline{\Omega}\}.
\]

\begin{lemma}\label{lem1}
Let $u\in H^2(\Omega \times (a,b))$ satisfying $u=0$ on $\partial \Omega \times (a,b)$ and let $\epsilon >0$. Then
\begin{equation}\label{b08}
\mathcal{E}_u(a)+\mathcal{E}_u(b)\le \left(\frac{2}{b-a}+\epsilon \right) \mathcal{E}_u+\frac{2}{\epsilon}\|Pu\|_{L^2(\Omega \times (a,b))}^2.
\end{equation}
\end{lemma}

\begin{proof}
 Since
\[
\mathcal{E}(a)=\mathcal{E}(t)+2\int_a^t\int_\Omega Pu(x,s)\partial_tu(x,s)dxds,
\]
we get by using an elementary convexity inequality
\begin{equation}\label{b09}
\mathcal{E}_u(a)\le \mathcal{E}_u(t)+\frac{1}{\epsilon}\|Pu\|_{L^2(\Omega \times (a,b))}^2+\epsilon \int_a^t\mathcal{E}_u(s)ds.
\end{equation}
We have similarly
\begin{equation}\label{b010}
\mathcal{E}_u(b)\le \mathcal{E}_u(t)+\frac{1}{\epsilon}\|Pu\|_{L^2(\Omega \times (a,b))}^2+\epsilon \int_t^b\mathcal{E}_u(s)ds.
\end{equation}
The sum, side by side, of \eqref{b09} and \eqref{b010} yields
\[
\mathcal{E}_u(a)+\mathcal{E}_u(b)\le 2\mathcal{E}(t)+\frac{2}{\epsilon}\|Pu\|_{L^2(\Omega \times (a,b))}^2+\epsilon \mathcal{E}_u.
\]
We obtain \eqref{b08} by integrating this inequality with respect to $t$, $t\in (a,b)$.
\end{proof}

\begin{lemma}\label{lem2}
Let $u\in H^2(\Omega \times (a,b))$ satisfying $u=0$ on $\partial \Omega \times (a,b)$ and let $\delta >0$. Then
\begin{equation}\label{b011}
\left| \left[\int_\Omega \partial_tu(\cdot,t)w(\cdot ,t)dx\right]_{t=a}^{t=b}\right|\le \overline{\delta}(\mathcal{E}_u(a)+\mathcal{E}_u(b)),
\end{equation}
with $w=2(m\cdot\nabla u)+(n-1)u$ and $\overline{\delta}=\max \left( \delta ,\mathfrak{d}_0^2/\delta\right)$.
\end{lemma}

\begin{proof}
We first prove
\begin{equation}\label{b11}
\|w(\cdot ,t)\|_{L^2(\Omega )}\le \|2m\cdot \nabla' u(\cdot, t)\|_{L^2(\Omega )}.
\end{equation}
We have
\begin{align*}
\|w(\cdot ,t)\|_{L^2(\Omega )}^2- \|2m&\cdot \nabla' u(\cdot, t)\|_{L^2(\Omega )}^2
\\
&=\int_\Omega \left[(n-1)^2u(\cdot ,t)^2+4(n-1)um\cdot \nabla' u(\cdot ,t)\right]dx
\\
&=\int_\Omega \left[(n-1)^2u^2(\cdot ,t)+2(n-1)m\cdot \nabla' u^2(\cdot ,t)\right]dx
\\
&=\int_\Omega \left[(n-1)^2u^2(\cdot ,t)-2(n-1)n u^2(\cdot ,t)\right]dx
\\
&=\int_\Omega (1-n^2)u^2(\cdot ,t)dx\le 0.
\end{align*}
This leads immediately to \eqref{b11}.

On the other hand, we get by applying Cauchy-Schwarz's inequality
\[
\left|\int_\Omega \partial_tu(\cdot,t)w(\cdot ,t)dx\right|\le \|\partial_tu(\cdot ,t)\|_{L^2(\Omega )}\|w(\cdot ,t)\|_{L^2(\Omega )}.
\]
This together with \eqref{b11} imply
\begin{align*}
\left|\int_\Omega \partial_tu(\cdot,t)w(\cdot ,t)dx\right|&\le \|\partial_tu(\cdot ,t)\|_{L^2(\Omega )}\|2m\cdot \nabla' u(\cdot, t)\|_{L^2(\Omega )}
\\
&\le \delta \|\partial_tu(\cdot ,t)\|_{L^2(\Omega )}^2+\frac{\mathfrak{d}_0^2}{\delta}\|\nabla' u(\cdot, t)\|_{L^2(\Omega )}^2
\\
&\le \overline{\delta}\mathcal{E}_u(t),
\end{align*}
 which gives \eqref{b011} in a straightforward manner.
\end{proof}

We obtain the following corollary by taking $\epsilon= 2/(b-a)$ in \eqref{b08} and $\delta =\mathfrak{d}_0$ in \eqref{b011}.
\begin{corollary}\label{cor1}
For $u\in H^2(\Omega \times (a,b))$ satisfying $u=0$ on $\partial \Omega \times (a,b)$, we have
\begin{equation}\label{b012}
\left| \left[\int_\Omega \partial_tu(\cdot,t)w(\cdot ,t)dx\right]_{t=a}^{t=b}\right|\le \frac{4\mathfrak{d}_0}{b-a}\mathcal{E}_u+(b-a)\mathfrak{d}_0\|Pu\|_{L^2(\Omega \times (a,b))}^2.
\end{equation}
\end{corollary}

As $B=(\nabla' a^{ij}\cdot m)$, we find
\[
\left|B(x)\xi \cdot \xi \right|\le \mathfrak{d}_0\varkappa |\xi|^2,\quad x\in \Omega ,\;\xi \in \mathbb{R}^n.
\]
Whence a simple computations enable us to get
\begin{align}
&\left| \int_a^b\int_\Omega B\nabla' u\cdot \nabla' udxdt\right| \le \varkappa \mathfrak{d}_0 \int_a^b\int_\Omega |\nabla' u|^2dxdt\le  \varkappa \kappa \mathfrak{d}_0\mathcal{E}_u ,\label{b013}
\\
&\left| \int_a^b\int_{\partial \Omega} (\partial_\nu u)^2(A\nu \cdot \nu)(m\cdot \nu)d\sigma(x)dt  \right| \le \kappa \mathfrak{d}_0 \int_a^b\int_{\partial \Omega} (\partial_\nu u)^2d\sigma(x)dt.\label{b014}
\end{align}

We get in light of \eqref{b11}, for $\delta >0$,
\begin{align}
\left| \int_a^b \int_\Omega Puwdxdt\right| \le &\|Pu\|_{L^2(\Omega \times (a,b))}\|2m\cdot \nabla' u\|_{L^2(\Omega \times (a,b))}\label{b015}
\\
&\le \frac{\mathfrak{d}_0}{\delta}\|Pu\|_{L^2(\Omega \times (a,b))}^2+\delta \mathfrak{d}_0\|\nabla' u\|_{L^2(\Omega \times (a,b))}^2\nonumber 
\\
&\le \frac{\mathfrak{d}_0}{\delta}\|Pu\|_{L^2(\Omega \times (a,b))}^2+\delta \kappa \mathfrak{d}_0\mathcal{E}_u.\nonumber
\end{align}

Then $\delta =1/(b-a)$ in \eqref{b015} yields
\begin{equation}\label{b016}
\left| \int_a^b \int_\Omega Puwdxdt\right| \le \frac{\kappa\mathfrak{d}_0}{b-a}\mathcal{E}_u+\mathfrak{d}_0(b-a)\|Pu\|_{L^2(\Omega \times (a,b))}^2.
\end{equation} 

Inequalities \eqref{b012}, \eqref{b013}, \eqref{b014} and \eqref{b016} in \eqref{b07} give
\begin{align}
\mathcal{E}_u\le \kappa \varkappa \mathfrak{d}_0\mathcal{E}_u &+\frac{(4+\kappa)\mathfrak{d}_0}{b-a}\mathcal{E}_u\label{b017.0}
\\
&+2\mathfrak{d}_0(b-a)\|Pu\|_{L^2(\Omega \times (a,b))}^2+\kappa\mathfrak{d}_0\|\partial_\nu u\|_{L^2(\partial \Omega \times (a,b))}^2.\nonumber
\end{align}

Thus, under the condition
\begin{equation}\label{B017.1}
\varrho_0= 1-\kappa \varkappa \mathfrak{d}_0>0,
\end{equation}
inequality \eqref{b017.0} yields
\begin{equation}\label{b018}
\mathcal{E}_u\le \frac{(4+\kappa )\mathfrak{d}_0}{\varrho_0(b-a)}\mathcal{E}_u+2\mathfrak{d}_0\frac{b-a}{\varrho_0}\|Pu\|_{L^2(\Omega \times (a,b))}^2+\frac{ \mathfrak{d}_0\kappa}{\varrho_0}\|\partial_\nu u\|_{L^2(\partial \Omega \times (a,b))}^2.
\end{equation}

\begin{remark}
{\rm
Let $A_0$ satisfying conditions \eqref{main1} and \eqref{main2}. Then \eqref{B017.1} holds for $A=\lambda A_0$, provided that $\lambda <(\kappa \varkappa \mathfrak{d}_0)^{-1/2}$.
}
\end{remark}

Next, fix $\tilde{\mathbf{c}}$ sufficiently large in such a way that
\[
\tilde{\varrho} =1-\frac{(4+\kappa )\mathfrak{d}_0}{\varrho_0\tilde{\mathbf{c}}}>0.
\]
In that case  we get from \eqref{b018}, for $\mathbf{c}\ge \tilde{\mathbf{c}}$,
\[
\mathcal{E}_u\le \frac{2\mathfrak{d}_0\mathbf{c}}{\varrho_0\tilde{\varrho}}\|Pu\|_{L^2(\Omega \times (a,b))}^2+\frac{ \mathfrak{d}_0 \kappa}{\varrho_0\tilde{\varrho}}\|\partial_\nu u\|_{L^2(\partial \Omega \times (a,b))}^2.
\]
We derive from this inequality and \eqref{b08} the following result.

\begin{proposition}\label{pr1}
Assume that $\kappa \varkappa \mathfrak{d}_0<1$ and fix $\tilde{\mathbf{c}}>0$ such that
\[
\tilde{\mathbf{c}}>\frac{(4+\kappa)\mathfrak{d}_0}{1-\kappa\varkappa \mathfrak{d}_0}.
\]
Then there exits a constant $C>0$, depending only on $n$, $\tilde{\mathbf{c}}$, $\mathfrak{d}_0$, $\varkappa$ and $\kappa$ so that, for any $a$, $b\in \mathbb{R}$ with $b-a=\mathbf{c}\ge \tilde{\mathbf{c}}$ and any $u\in H^2(\Omega \times (a,b))$ satisfying $u=0$ on $\partial \Omega\times (a,b)$, we have
\[
\|\nabla u(\cdot ,a)\|_{L^2(\Omega )}+\|\nabla u(\cdot ,b)\|_{L^2(\Omega )}\le C\left( \|Pu\|_{L^2(\Omega \times (a,b))}+\|\partial _\nu u\|_{L^2(\partial \Omega \times (a,b))}\right).
\]
\end{proposition}

In light of inequality \eqref{a1}, we have as a consequence of Proposition \ref{pr1} the following result.
\begin{corollary}\label{cor2}
Assume that $\kappa \varkappa \mathfrak{d}_0<1$ and let $\tilde{\mathbf{c}}>0$ satisfies
\[
\tilde{\mathbf{c}}>\frac{(4+\kappa)\mathfrak{d}_0}{1-\kappa\varkappa \mathfrak{d}_0}.
\]
Then there exits a constant $C>0$, depending only on $n$, $\tilde{\mathbf{c}}$, $\mathfrak{d}_0$, $\varkappa$ and $\kappa$ so that, for any $a$, $b\in \mathbb{R}$ with $b-a=\mathbf{c}\ge \tilde{\mathbf{c}}$ and any $u\in H^2((a,b),H^2(\Omega ))$, we have
\begin{align*}
\|\nabla u(\cdot ,a)\|_{L^2(\Omega )}&+\|\nabla u(\cdot ,b)\|_{L^2(\Omega )}
\\
&\le C\left( \|Pu\|_{L^2(\Omega \times (a,b))}+\|\partial _\nu u\|_{L^2(\partial \Omega \times (a,b))}+\|u\|_{H^2((a,b),H^{3/2}(\partial \Omega ))}\right).
\end{align*}
\end{corollary}

Finally, we apply the last corollary to $u$ and $\partial _tu$, respectively with $(a,b)=\left(-T,-(1-\overline{\delta}/2)T\right)$ and $(a,b)=\left((1-\overline{\delta}/2)T,T\right)$, and we use that $P$ and $\partial _t$ commute. We get
\begin{corollary}\label{cor3}
Assume that
\[
\kappa \varkappa \mathfrak{d}_0<1,\quad T_0> \sqrt{8}\, \overline{\rho} >\frac{2(4+\kappa)\mathfrak{d}_0}{1-\kappa\varkappa \mathfrak{d}_0}.
\]
Then there exits a constant $C>0$, only depending  on $n$, $\mathfrak{d}_0$, $\varkappa$, $\kappa$, $T_0$ and $\overline{\rho}$, so that,  for any $T\ge T_0$ and $u\in H^3(I_T,H^2(\Omega ))$, we have
\begin{align*}
&\sum_{\epsilon \in \{-,+\}}\sum_{\ell =0}^1\|\nabla \partial_t^\ell u(\cdot ,\epsilon T)\|_{L^2(\Omega )}\le 
\\
&\hskip 1cm C\left( \|Pu\|_{H^1\left(J_{\overline{\delta}/2},L^2(\Omega )\right)}+\|\partial _\nu u\|_{H^1\left(J_{\overline{\delta}/2},L^2(\partial \Omega )\right)}+\|u\|_{H^3\left(J_{\overline{\delta}/2},H^{3/2}(\partial \Omega )\right)}\right).
\end{align*}
\end{corollary}

Define
\begin{align*}
\tilde{\mathfrak{D}}_{\overline{\delta}}(u)&=\|Pu\|_{H^1\left(I_T, L^2(\Omega )\right)}+\|u\|_{H^{1,1}\left(\Sigma_T\right)}+\|\partial_\nu u\|_{L^2\left(\Sigma_T\right)}
\\
&\hskip 1cm+\|u\|_{H^3\left( J_{\overline{\delta}/2},H^{3/2}(\partial \Omega)\right)}+\|\partial_\nu u\|_{H^1\left( J_{\overline{\delta}/2},L^2(\partial \Omega)\right)}.
\end{align*}

The following theorem follows by combining Theorem \ref{theorem+3.1} and Corollary \ref{cor3}.

\begin{theorem}\label{theorem3.1}
Let $s\in (0,1/2)$  and assume that
\[
\kappa \varkappa \mathfrak{d}_0<1,\quad T_0> \sqrt{8}\, \overline{\rho} >\frac{2(4+\kappa)\mathfrak{d}_0}{1-\kappa\varkappa \mathfrak{d}_0}.
\]
Let $0<\underline{\rho}<\overline{\rho}$. Then there exist three constants $C>0$, $c>0$ and $\gamma >0$, only depending on $\Omega$, $\alpha$, $T_0$, $\kappa$ $\varkappa$, $s$, $\underline{\rho}$ and $\overline{\rho}$  so that, for any $T\ge T_0$, $u\in C^{1,\alpha}(\overline{Q}_T)\cap H^3(I_T, H^2(\Omega ))$ and $\delta \in [\sqrt{8}\, \underline{\rho}/T, \sqrt{8}\, \overline{\rho}/T]$, we have
\[
C\|u\|_{L^2\left(I_T, H^1(\Omega )\right)} \le \delta ^s\|u\|_{\mathcal{X}_T} + \mathfrak{e}_2(c\delta^{-14})\tilde{\mathfrak{D}}_{\overline{\delta}}(u),
\] 
where $\overline{\delta}=\sqrt{8}\, \overline{\rho}/T$.
\end{theorem}

Define, for $\beta >0$ and $\varrho_0\le e^{-1}$,
\[
\Theta_{\varrho_0,\beta}(\varrho )=
\left\{
\begin{array}{ll}
(\ln |\ln \varrho |)^{-\beta } \quad &\mbox{if}\; 0<\varrho\le \varrho_0 ,
\\
\varrho &\mbox{if}\; \varrho \ge \varrho_0,
\end{array}
\right.
\]
that we extend by continuity at $\varrho=0$ by setting $\Theta_{\varrho_0,\beta}(0)=0$.

We can mimic the proof of Corollary \ref{corollary2++} in order to obtain the following result.
\begin{corollary}\label{corollary3++}
Let $\mathfrak{b}>0$. Under the assumptions of Theorem \ref{theorem3.1}, there exists $\hat{T}\ge T_0$, only depending of $\Omega$, $\alpha$, $T_0$, $\kappa$, $\varkappa$, $s$, $\alpha$, $\underline{\rho}$ and $\overline{\rho}$ and $\mathfrak{b}$,  so that, for any $T\ge \hat{T}$,  and $u\in C^{1,\alpha}(\overline{Q}_T)\cap H^3(I_T, H^2(\Omega ))$ satisfying $u\ne 0$ and
\[
\frac{\tilde{\mathfrak{D}}_{\overline{\delta}}(u)}{\|u\|_{\mathcal{X}_T}} \le \mathfrak{b},
\]
we have
\[
C\|u\|_{L^2\left(I_T, H^1(\Omega )\right)} \le \|u\|_{\mathcal{X}_T}\Theta_{\varrho_0,s/8}\left(\mathfrak{b} \right),
\]
where the constants $C>0$ and  $0<\varrho_0\le e^{-1}$ only depend of $\Omega$, $\alpha$, $T_0$, $\kappa$, $\varkappa$, $\alpha$, $s$, $\underline{\rho}$, $\overline{\rho}$, $\mathfrak{b}$ and $T$.
\end{corollary}

\appendix
\section{The Cauchy problem for elliptic equations}\label{app1}

Throughout this appendix, $D$ denotes a bounded domain of $\mathbb{R}^d$, $d\ge 2$, with Lipschitz boundary.

We adopt in the sequel Einstein's summation convention. Each  index  appearing both up and down  has to be summed from $1$ to $d$.

\subsection{Carleman and Caccioppoli inequalities}

Let $\mathcal{R}$ be an arbitrary set and consider the family of operators $(L_r)$, for each $r\in \mathcal{R}$, acting on $D$ as follows 
\[
L_ru=\mbox{div}(A_r\nabla u),
\]
where, for each $r\in \mathcal{R}$, $A_r=(a_r^{ij})$ is a symmetric matrix with  coefficients in $W^{1,\infty}(D)$ satisfying
 
\begin{equation}\label{1.1.0}
\kappa ^{-1}|\xi |^2 \le A_r(x)\xi \cdot \xi \le \kappa |\xi |^2,\quad  x\in \Omega , \; \xi \in \mathbb{R}^d\; \mbox{and}\; r\in \mathcal{R},
\end{equation}
and
\begin{equation}\label{1.2.0}
\sum_{k=1}^d\left|\partial_k a_r^{ij}(x)\xi _i\xi_j\right| \le \varkappa |\xi|^2,\quad x\in \Omega ,\; \xi \in \mathbb{R}^d,\; r\in \mathcal{R},
\end{equation}
for some constants $\kappa \ge 1$ and $\varkappa \ge 0$.

Pick $0\le \psi \in C^2(\overline{D})$ without critical points in $\overline{D}$ and let $\varphi =e^{\lambda \psi}$. Then the following Carleman inequality was proved in \cite{Ch2016} (see also \cite{Ch2018} and \cite{BL} for operators with complex coefficients).

\begin{theorem}\label{theorem1.1} 
There exist three positive constants $C$, $\lambda _0$ and $\tau _0$, only depending on $\psi$, $D$, $\kappa$ and $\varkappa$, so that
\begin{align}
C\int_D &\left(\lambda ^4\tau ^3\varphi ^3v^2+\lambda ^2\tau \varphi |\nabla v|^2 \right)e^{2\tau \varphi} dx \nonumber
\\
&\le \int_D (L_rv)^2e^{2\tau \varphi}dx+\int_{\partial D} \left( \lambda^3\tau ^3\varphi ^3v^2+\lambda \tau \varphi |\nabla v|^2\right)e^{2\tau \varphi} d\sigma(x) ,\label{1.3}
\end{align}
for all $r\in \mathcal{R}$ and $v\in H^1(D)$ satisfying $L_rv\in L^2(D)$, $\lambda \geq \lambda _0$ and $\tau \geq \tau _0$.
\end{theorem}

The following Caccioppoli type inequality  will be useful in the sequel.

\begin{lemma}\label{lemma1.1}
Let $0<k<\ell$. There exists a constant $C>0$,  only depending on $D$, $k$, $\ell$, $\kappa$ and $\varkappa$, so that, for any $x\in D$, $0<\rho<\mbox{dist}(x,\partial D )/\ell$ and $u\in H^1(D)$ satisfying $L_ru\in L^2(D)$, for some $r\in \mathcal{R}$, we have
\begin{equation}\label{1.4}
C\int_{B(x,k\rho )}|\nabla u|^2dy\leq \frac{1}{\rho^2}\int_{B(x,\ell \rho)}u^2dy+\int_{B(x,\ell \rho)}(L_ru)^2dy.
\end{equation}
\end{lemma}

\begin{proof}
We give the proof when $k=1$ and $\ell =2$. The proof for arbitrary $k$ and $\ell$ is quite similar.

\smallskip
Let $x\in D$, $0<\rho<\mbox{dist}(x,\partial D )/2$, $r\in \mathcal{R}$ and $u\in H^1(D )$ satisfying $L_ru\in L^2(D)$. Then
\begin{equation}\label{1.5}
\int_D a_r^{ij}\partial _i u\partial _jv dy =-\int_DL_ruvdy,\quad v\in C_0^1(D).
\end{equation}
Pick $\chi \in C_0^\infty (B(x,2\rho))$ so that $0\leq \chi \leq 1$, $\chi =1$ in a neighborhood of $B(x,\rho)$ and $|\partial ^\gamma \chi|\leq c\rho^{-|\gamma |}$ for $|\gamma |\leq 2$, where $c$ is a universal constant. Therefore identity \eqref{1.5} with $v=\chi u$ gives
\begin{align*}
\int_D \chi a_r^{ij}\partial _iu\partial _ju dy&=-\int_D u a_r^{ij}\partial _iu\partial _j \chi dy -\int_D\chi uL_rudy
\\
&=-\frac{1}{2}\int_D  a_r^{ij}\partial _iu^2\partial _j \chi dy -\int_D\chi uL_rudy
\\
&=\frac{1}{2}\int_D  u^2\partial _i\left( a_r^{ij}\partial _j \chi \right) dy -\int_D\chi uL_rudy.
\end{align*}
But
\[
\int_D \chi a_r^{ij}\partial _iu\partial _ju dy\geq \kappa \int_D \chi |\nabla u|^2dy.
\]
Whence
\[
C \int_{B(x,\rho)} |\nabla u|^2dy\le \frac{1}{\rho ^2}\int_{B(x,2\rho)}u^2dy+\int_{B(x,2\rho)}(L_ru)^2dy,
\]
as expected.
\end{proof}

\subsection{Three-ball inequalities}

Consider the elliptic operator $L$ acting on $D$ as follows
\[
Lu=\mbox{div}(A\nabla u),
\]
where $A=(a^{ij})$ is a symmetric matrix with  coefficients in $W^{1,\infty}(D)$ satisfying 
\begin{equation}\label{1.6.0}
\kappa ^{-1}|\xi |^2 \le A(x)\xi \cdot \xi \le \kappa |\xi |^2,\quad  x\in D , \; \xi \in \mathbb{R}^d,
\end{equation}
and
\begin{equation}\label{1.7.0}
\sum_{k=1}^d\left|\partial_k a^{ij}(x)\xi _i\xi_j\right| \le \varkappa |\xi|^2,\quad x\in D ,\; \xi \in \mathbb{R}^d.
\end{equation}
for some constants $\kappa \ge 1$ and $\varkappa>0$.

\begin{theorem}\label{theorem1.2}
Let $0<k<\ell<m$. There exist $C>0$ and $0<\gamma <1$,  only depending on  $D$, $k$, $\ell$, $m$, $\kappa$ and $\varkappa$,  so that, for any  $v\in H^1(D )$ satisfying $Lv\in L^2(D)$, $y\in D$ and $0<r< \mbox{dist}(y,\partial D)/m$, we have
\[
C\|v\|_{L^2(B(y,\ell r))}\le \left(\|v\|_{L^2(B(y,kr))}+\|Lv\|_{L^2(B(y,m r))}\right)^\gamma \|v\|_{L^2(B(y,m r))}^{1-\gamma}.
\]
\end{theorem}

\begin{proof}
As in the preceding lemma, we give the proof when $k=3/2$, $\ell =2$ and $m =7/2$. The proof of arbitrary $k$, $\ell$ and $m$ is similar. 

\smallskip
Let $v\in H^1(D)$ satisfying $Lv\in L^2(D)$ and set $B(s)=B(0,s)$, $s>0$. Fix $y\in D$ and  
\[
0<r< r_y=2\mbox{dist}(y,\partial D )/7 \left( \le 2\mbox{diam}(D) /7\right).
\]
 Let
\[
w(x)=v(rx+y),\; x\in B(3).
\]
Then straightforward computations show that  
\begin{equation}\label{1.8}
L_rw=\textrm{div}(A_r\nabla w )=r^2Lv(rx+y)\;\; \mbox{in}\; B(7/2),
\end{equation}
where 
\[
A_r(x)=(a_r^{ij}(x)),\quad a_r^{ij}(x)=a^{ij}(rx+y).
\]

It is not hard to see that the family $(A_r)$  satisfies \eqref{1.6.0} and \eqref{1.7.0}, uniformly with respect to $r\in (0,r_y)$.

\smallskip
Set
\[
U=\left\{x\in \mathbb{R}^n;\; 1/2<|x|<3\right\},\quad K=\left\{x\in \mathbb{R}^n;\; 1\leq |x|\leq 5/2\right\}
\]
and pick $\chi \in C_0^\infty (U)$ satisfying $0\le \chi \le 1$ and  $\chi =1$ in a neighborhood of $K$.

\smallskip
We get by applying Theorem \ref{theorem1.1} to $\chi w$, with $D$ is substituted by $U$, for $\lambda \geq \lambda _0$ and $\tau \geq \tau _0$,
\begin{align}
&C\int_{B(2)\setminus B(1)} \left (\lambda ^4\tau ^3\varphi ^3w^2+\lambda ^2\tau \varphi |\nabla w|^2 \right)e^{2\tau \varphi} dx \label{1.9}
\\
&\hskip 5cm \leq \int_{B(3)} (L_r(\chi w))^2e^{2\tau \varphi}dx. \nonumber
\end{align}
We have $L_r(\chi w)=\chi L_rw+Q_r(w)$, with
\[
Q_r(w)=\partial_j \chi a_r^{ij}\partial_iw+\partial_j(a_r^{ij}w)\partial_iw+a_r^{ij}\partial_{ij}^2\chi w,
\]
\[
\mathrm{supp}(Q_rw)\subset \left\{1/2 \le |x|\le 1\right\}\cup \left\{5/2\le |x|\leq 3\right\}
\]
and
\[
(Q_rw)^2 \le \Lambda (w^2+|\nabla w|^2 ),
\]
the constant $\Lambda $ is independent of $r$. Therefore, fixing $\lambda$ and changing $\tau _0$ if necessary, \eqref{1.8} yields, for $\tau \geq \tau _0$,
\begin{align}
&C\int_{B(2)} \left (w^2+|\nabla w|^2 \right)e^{2\tau \varphi} dx\leq \int_{B(1)} \left (w^2+|\nabla w|^2 \right)e^{2\tau \varphi} dx \label{1.10}
\\
&\hskip 1cm +\int_{B(3)}(L_rw)^2e^{2\tau \varphi} dx+\int_{\left\{5/2\leq |x|\leq 3\right\}} \left (w^2+|\nabla w|^2 \right)e^{2\tau \varphi} dx.\nonumber
\end{align}

We get by taking $\psi (x)=9-|x|^2$ in \eqref{1.10}, which is without critical points in $U$, for $\tau \geq \tau _0$,
\begin{align}
&C\int_{B(2)} \left (w^2+|\nabla w|^2 \right) dx \leq e^{-\beta \tau}\int_{B(3)} \left (w^2+|\nabla w|^2 \right)dx\label{1.11}
\\
&\hskip 2cm+ e^{\alpha \tau}\left[\int_{B(1)} \left(w^2+|\nabla w|^2 \right) dx+\int_{B(3)}(L_rw)^2dx\right], \nonumber 
\end{align}
where
\[
\alpha =\left(e^{9\lambda}-e^{5\lambda}\right),\quad \beta =\left(e^{5\lambda}-e^{11\lambda/4}\right) .
\]

On the other hand, we have from Caccioppoli's inequality \eqref{1.4}
\begin{align}
&C\int_{B(1)} |\nabla w|^2 dx\le \int_{B\left(3/2\right)} w^2 dx+\int_{B\left(3/2\right)}(L_rw)^2dx, \label{1.12}
\\
&C\int_{B(3)} |\nabla w|^2 dx\le \int_{B\left(7/2\right)} w^2 dx+\int_{B\left(7/2\right)}(L_rw)^2dx .\label{1.13}
\end{align}
Inequalities \eqref{1.12} and \eqref{1.13} in \eqref{1.11} give
\begin{align}
C\int_{B(2)} w^2dx&\le e^{\alpha \tau}\left[\int_{B\left(3/2\right)} w^2 dx+\int_{B\left(7/2\right)}(L_rw)^2dx\right]\label{1.14}
\\
&\hskip 4cm+e^{-\beta \tau}\int_{B\left(7/2\right)} w^2 dx.\nonumber
\end{align}

We proceed as in the last part of the proof of \cite[Theorem 2.17, page 21]{Ch2016} in order to derive from \eqref{1.14}
\begin{equation}\label{1.15}
C\|w\|_{L^2(B(2))}\le \left(\|w\|_{L^2\left(B\left(3/2\right)\right)}+\|L_rw\|_{L^2\left(B\left(7/2\right)\right)}\right)^\gamma \|w\|_{L^2\left(B\left(7/2\right)\right)}^{1-\gamma},
\end{equation}
with $\gamma =\alpha/(\alpha +\beta)$.

We obtain in a straightforward manner from \eqref{1.15}
\[
C\|v\|_{L^2(B(y,2r))}\le \left(\|v\|_{L^2\left(B\left(y,3r/2\right)\right)}+\|Lv\|_{L^2\left(B\left(y,7r/2\right)\right)}\right)^\gamma \|v\|_{L^2\left(B\left(y,7r/2\right)\right)}^{1-\gamma}.
\]
This is the expected inequality.
\end{proof}

Prior to state the three-ball inequality for the gradient, we demonstrate a generalized Poincar\'e-Wirtinger type inequality. To this end, if $\mathcal{O}$ is an arbitrary open bounded subset of $\mathbb{R}^d$, $f\in L^2(\mathcal{O})$ and $E\subset \mathcal{O}$ is Lebesgue-measurable set with non zero Lebesgue measure $|E|$, define
\[
M_E(f)=\frac{1}{|E|}\int_E f(x)dx.
\]

The following lemma can be deduced from \cite[Theorem 1]{Me}. For sake of completeness, we provide here a simple proof.

\begin{lemma}\label{lemmaGPW}
Let $\mathcal{O}$ is an arbitrary open bounded subset of $\mathbb{R}^d$. There exists a constant $\aleph$, only depending on $\mathcal{O}$, so that, for any $f\in L^2(\mathcal{O})$ and $E\subset \mathcal{O}$ Lebesgue-measurable set with non zero Lebesgue measure $|E|$, we have
\begin{equation}\label{1.16}
\|f-M_E(f)\|_{L^2(\mathcal{O})}\le \aleph \frac{|\mathcal{O}|^{1/2}}{|E|^{1/2}}\|\nabla f\|_{L^2(\mathcal{O})}.
\end{equation}
\end{lemma}

\begin{proof}
A simple application of Cauchy-Schwarz's inequality gives
\begin{equation}\label{pw1}
\|M_E(f)\|_{L^2(\mathcal{O})}\le \frac{|\mathcal{O}|^{1/2}}{|E|^{1/2}}\|f\|_{L^2(\mathcal{O})}.
\end{equation}
Inequality \eqref{pw1}, with $E=\mathcal{O}$ and $f$ substituted by $f-M_E(f)$, yields
\begin{equation}\label{pw2}
\|M_{\mathcal{O}}(f-M_E(f))\|_{L^2(\mathcal{O})}\le \|f-M_E(f)\|_{L^2(\mathcal{O})}.
\end{equation}
On the other hand, by the classical Poincar\'e-Wirtinger's inequality, there exists a constant $\aleph$, only depending on $\mathcal{O}$, so that
\begin{equation}\label{pw3}
\|f-M_{\mathcal{O}}(f)\|_{L^2(\mathcal{O})}\le (\aleph /2)\|\nabla f\|_{L^2(\mathcal{O})}.
\end{equation}
Now, as $M_E (M_{\mathcal{O}}(f))=M_{\mathcal{O}}(f)$, we have
\[
f-M_E(f)=f-M_{\mathcal{O}}(f)-M_E(f-M_{\mathcal{O}}(f)).
\]
We then obtain in light of \eqref{pw1}
\[
\|f-M_E(f)\|_{L^2(\mathcal{O})}\le \left(1+\frac{|\mathcal{O}|^{1/2}}{|E|^{1/2}}\right) \|f-M_{\mathcal{O}}(f)\|_{L^2(\mathcal{O})}
\]
implying 
\[
\|f-M_E(f)\|_{L^2(\mathcal{O})}\le 2\frac{|\mathcal{O}|^{1/2}}{|E|^{1/2}}\|f-M_{\mathcal{O}}(f)\|_{L^2(\mathcal{O})}.
\]
Therefore, in light of \eqref{pw3}, we get 
\[
\|f-M_E(f)\|_{L^2(\mathcal{O})}\le \aleph \frac{|\mathcal{O}|^{1/2}}{|E|^{1/2}}\|\nabla f\|_{L^2(\mathcal{O})}
\]
as expected.
\end{proof}

 \begin{theorem}\label{theorem1.3}
Let $0<k<\ell<m$. There exist $C>0$ and $0<\gamma <1$, only depending on  $D$, $k$, $\ell$, $m$, $\kappa$ and $\varkappa$,  so that, for any  $v\in H^1(D)$ satisfying $Lv\in L^2(D)$, $y\in D$ and $0<r< \mbox{dist}(y,\partial D)/m$.
\[
C\|\nabla v\|_{L^2(B(y,\ell r))}\le \left(\|\nabla v\|_{L^2(B(y,kr))}+\|Lv\|_{L^2(B(y,m r))}\right)^\gamma \|\nabla v\|_{L^2(B(y,m r))}^{1-\gamma}.
\]
\end{theorem}
 
 \begin{proof}
 We keep the notations of the proof of Theorem \ref{theorem1.2}. We apply the generalized Poincar\'e-Wirtinger's inequality \eqref{1.16} in order to obtain, where
 \[
 \varrho =\frac{1}{|B(1)|}\int_{B(1)}w(x)dx,
 \]
 \begin{align}
 &\int_{B(1)}(w-\varrho)^2 dx\le C\int_{B(1)}|\nabla w|^2dx, \label{1.17}
 \\
 &\int_{B(3)}(w-\varrho)^2 dx\le C\int_{B(3)}|\nabla w|^2dx. \label{1.18}
 \end{align}
On the other hand \eqref{1.11}, in which $w$ is substituted by $w-\varrho$, gives
 \begin{align}
C\int_{B(2)} |\nabla w|^2 dx &\le e^{-\beta \tau}\int_{B(3)} \left ((w-\varrho)^2+|\nabla w|^2 \right)dx\label{1.19}
\\
&+e^{\alpha \tau}\left[\int_{B(1)} \left((w-\varrho)^2+|\nabla w|^2 \right) dx+\int_{B(3)}(L_rw)^2dx\right]. \nonumber 
\end{align}
Using \eqref{1.17} and \eqref{1.18} in \eqref{1.19}, we get
\[
C\int_{B(2)} |\nabla w|^2 dx \leq e^{\alpha \tau}\left[\int_{B(1)} |\nabla w|^2 dx+\int_{B(3)}(L_rw)^2dx\right]+e^{-\beta \tau}\int_{B(3)} |\nabla w|^2 dx.
\]
The rest of the proof is identical to that of Theorem \ref{theorem1.2}.
 \end{proof}

\subsection{Stability of the Cauchy problem}

We shall use in the sequel the following lemma.

\begin{lemma}\label{lemma1.1+}
Let $(\eta _k)$ be a sequence of real numbers satisfying $0<\eta _k\le 1$, $k\in \mathbb{N}$, and
\[
\eta _{k+1}\le c(\eta_k +b )^\gamma ,\quad k\in \mathbb{N},
\]
for some constants $0<\gamma <1$, $b>0$ and $c\ge 1$. Then
\begin{equation}\label{1.20}
\eta _k\le C(\eta_0+b)^{\gamma ^k},
\end{equation}
where $C=(2c)^{\frac{1}{1-\gamma}}$.
\end{lemma}

\begin{proof}
If $\eta_0 +b\ge 1$,  \eqref{1.20} is trivially satisfied. Assume then that $\eta _0+b< 1$. As 
\[
b< cb^\gamma < c(\eta _k +b)^\gamma ,\quad  k\in \mathbb{N},
\]
we obtain
\begin{equation}\label{1.21}
\eta_{k+1}+b\le 2c(\eta _k +b)^\gamma.
\end{equation}
If $\tau_k=\eta _k +b$, \eqref{1.21} can rewritten as follows
\[
\tau_{k+1} \le 2c \tau_k^\gamma ,\quad k\in \mathbb{N}.
\]
A simple induction  in $k$ yields
\[
\eta _k\le (2c)^{1+\gamma +\ldots +\gamma ^{k-1}}\tau_0^{\gamma ^k}\le (2c)^{\frac{1}{1-\gamma}}(\eta_0+b)^{\gamma ^k}.
\]
The proof is then complete.
\end{proof}

Note that, as $D$ is Lipschitz, it has the uniform cone property (we refer for instance to \cite{HP} for more details and a proof). In particular, there exist $R>0$ and $\theta \in ]0,\pi /2[$ so that, to any $\tilde{x}\in \partial D$ we find $\xi =\xi (\tilde{x})\in \mathbb{S}^{d-1}$ with the property that
\[
\mathcal{C}(\tilde{x})=\left\{x\in \mathbb{R}^d;\; 0<|x-\tilde{x}|<R,\; (x-\tilde{x})\cdot \xi >|x-\tilde{x}|\cos \theta \right\}\subset D .
\]

\begin{proposition}\label{proposition1.1}
Let $0<\alpha \le 1$. There exist three constants $C>0$, $c>0$, $\beta >0$ and $\omega \Subset D$, only depending on $D$, $\kappa$ and $\varkappa$, so that:

\smallskip
\noindent
(1) for any $u\in H^1(D)\cap C^{0,\alpha}(\overline{D})$ with $Lu\in L^2(D)$ and $0<\epsilon<1$, we have
\[
C\|u\|_{L^\infty (\partial D)} \le e^{c/\epsilon}\left(\|u\|_{L^2(\omega )}+ \|Lu\|_{L^2(D)}\right)+\epsilon ^{\beta}\left([u]_\alpha +\|u\|_{L^2(D)}\right), \
\]

\noindent
(2) for any $u\in C^{1,\alpha}(\overline{D})$ with $Lu\in L^2(D)$ and $0<\epsilon<1$, we have
\[
C\|\nabla u\|_{L^\infty (\partial D)} \le e^{c/\epsilon}\left(\|\nabla u\|_{L^2(\omega )}+ \|Lu\|_{L^2(D)}\right)+\epsilon ^{\beta}\left([\nabla u]_\alpha +\|u\|_{L^2(D)}\right).
\]
Here $[\nabla u]_\alpha=\left(\sum_{i=1}^d[\partial_i u]_\alpha ^2\right)^{1/2}$.
\end{proposition}

\begin{proof}
Fix $\tilde{x}\in \partial D$ and let $\xi =\xi (\tilde{x})$ be as in the above consequence of the uniform cone property. Let $x_0=\widetilde{x}+(R/2)\xi$, $\delta_0=|x_0-\tilde{x}|$ and $\rho_0=(\delta_0/3) \sin \theta$. It is worth noting that $B(x_0,3\rho_0)\subset \mathcal{C}(\tilde{x})$. 

\smallskip
Define the sequence of balls $(B(x_k, 3\rho _k))$ as follows
\begin{eqnarray*}
\left\{
\begin{array}{ll}
x_{k+1}=x_k-\alpha _k \xi ,
\\
\rho_{k+1}=\mu \rho_k ,
\\
\delta_{k+1}=\mu \delta_k,
\end{array}
\right.
\end{eqnarray*}
where
\[
\delta_k=|x_k-\tilde{x}|,\;\; \rho _k=\varpi\delta_k,\;\; \alpha _k=(1-\mu)\delta_k ,
\]
with
\[
\varpi =\frac{\sin \theta}{3},\;\; \mu =\frac{3-2\sin \theta}{3-\sin \theta}.
\]
This definition guarantees that, for each $k$, $B(x_k,3\rho _k)\subset \mathcal{C}(\tilde{x})$ and
\begin{equation}\label{1.22}
B(x_{k+1},\rho _{k+1})\subset B(x_k,2\rho _k).
\end{equation}
Let $0\ne u\in H^1(D)\cap C^{0,\alpha}(\overline{D})$ with $Lu\in L^2(D)$. From Theorem \ref{theorem1.2}, we have
\[
\|u\|_{L^2(B(x_k,2\rho _k))}\le C\|u\|_{L^2(B(x_k,3\rho _k))}^{1-\gamma }\left(\|u\|_{L^2(B(x_k,\rho _k))}+\|Lu\|_{L^2(B(x_k,3\rho _k))}\right)^\gamma
\]
and then
\[
\|u\|_{L^2(B(x_k,2\rho _k))}\le C\|u\|_{L^2(D)}^{1-\gamma }\left(\|u\|_{L^2(B(x_k,\rho _k))}+\|Lu\|_{L^2(D)}\right)^\gamma .
\]
But $B(x_{k+1},\rho _{k+1})\subset B(x_k,2\rho _k)$. Hence
\[
\|u\|_{L^2(B(x_{k+1},\rho _{k+1}))}\le C\|u\|_{L^2(D)}^{1-\gamma }\left(\|u\|_{L^2(B(x_k,\rho _k))}+\|Lu\|_{L^2(D)}\right)^\gamma .
\]
Or equivalently
\[
\frac{\|u\|_{L^2(B(x_{k+1},\rho _{k+1}))}}{\|u\|_{L^2(D)}}\le C \left( \frac{\|u\|_{L^2(B(x_k,\rho _k))}}{\|u\|_{L^2(D)}}+\frac{\|Lu\|_{L^2(D)}}{\|u\|_{L^2(D)}}\right)^\gamma .
\]
Substituting if necessary $C$ by $\max(C,1)$, we may assume that $C\ge 1$. We can then apply Lemma \ref{lemma1.1+} in order to get
\[
\frac{\|u\|_{L^2(B(x_k,\rho _k))}}{\|u\|_{L^2(D)}}\le C \left( \frac{\|u\|_{L^2(B(x_0,\rho _0))}}{\|u\|_{L^2(D)}}+\frac{\|Lu\|_{L^2(D)}}{\|u\|_{L^2(D)}}\right)^{\gamma ^k}.
\]
This inequality can be rewritten in the following form
\begin{equation}\label{1.23}
\|u\|_{L^2(B(x_k,\rho _k))}\le C\left( \|u\|_{L^2(B(x_0,\rho _0))}+\|Lu\|_{L^2(D)}\right)^{\gamma ^k}\|u\|_{L^2(D)}^{1-\gamma^k}.
\end{equation}
Applying Young's inequality, we obtain, for $\epsilon >0$,
\begin{equation}\label{1.24}
C\|u\|_{L^2(B(x_k,\rho _k))}\le \epsilon^{-\frac{1}{\gamma^k}} \left(\|u\|_{L^2(B(x_0,\rho _0))}+\|Lu\|_{L^2(D)}\right)+\epsilon^{\frac{1}{1-\gamma ^k}}\|u\|_{L^2(D)}.
\end{equation}

Now, using that $u$ is H\"older continuous, we get
\[
|u(\tilde{x})|\le [u]_\alpha |\tilde{x}-x|^\alpha +|u(x)|,\;\; x\in B(x_k,\rho _k).
\]
Whence
\[
|\mathbb{S}^{d-1}|\rho_k^d|u(\tilde{x})|^2\le 2[u]_\alpha^2\int_{B(x_k,\rho _k)} |\tilde{x}-x|^{2\alpha}dx+ 2\int_{B(x_k,\rho _k)} |u(x)|^2dx,
\]
or equivalently 
\[
|u(\tilde{x})|^2\le 2|\mathbb{S}^{d-1}|^{-1}\rho_k^{-d}\left([u]_\alpha^2\int_{B(x_k,\rho _k)} |\tilde{x}-x|^{2\alpha}dx+ \int_{B(x_k,\rho _k)} |u(x)|^2dx\right).
\]
As $\delta_k=\mu ^k\delta_0$, we have 
\[
|\tilde{x}-x|\leq |\tilde{x}-x_k|+|x_k-x|\le \delta_k+\rho_k=(1+\varpi )\delta_k=(1+\varpi )\mu ^k\delta _0.
\]
Therefore
\begin{align}
|u(\tilde{x})|^2 &\le 2[u]_\alpha^2(1+\varpi )^\alpha \delta_0^\alpha\mu ^{2\alpha k} \label{1.25}
\\
&\hskip 2cm  +2|\mathbb{S}^{d-1}|^{-1}(\varpi \delta_0)^{-d}\mu ^{-dk}\|u\|_{L^2(B(x_k,\rho _k))}^2. \nonumber
\end{align}
Let
\[
\omega =\bigcup_{\tilde{x}\in \partial D}B(x_0(\tilde{x}),\rho _0)
\]
and introduce the following temporary notations
\begin{align*}
&M=[u]_\alpha +\|u\|_{L^2(D)},
\\
&N=\|u\|_{L^2(\omega )}+ \|Lu\|_{L^2(D)}.
\end{align*}
Then \eqref{1.25} yields
\begin{equation}\label{1.26}
C|u(\tilde{x})|\le M\mu ^{\alpha k}+\mu ^{-\frac{dk}{2}}\|u\|_{L^2(B(x_k,\rho _k))}.
\end{equation}
A combination of \eqref{1.24} and \eqref{1.26} implies
\[
C\|u\|_{L^\infty (\partial D)} \le \mu ^{-\frac{dk}{2}}\epsilon^{-\frac{1}{\gamma^k}}N+ \left(\mu ^{\alpha k}+\mu ^{-\frac{dk}{2}}\epsilon^{\frac{1}{1-\gamma ^k}}\right)M,\quad\epsilon >0.
\]
We take in this inequality  $\epsilon >0$ in such a way that $\mu ^{\alpha k}=\mu ^{-\frac{dk}{2}}\epsilon^{\frac{1}{1-\gamma ^k}}$. That is, $\epsilon =\mu^{\left( \frac{d}{2}+\alpha \right)k(1-\gamma ^k)}$. We obtain
\[
C\|u\|_{L^\infty (\partial D)} \le \mu^{\alpha k-\frac{k}{\gamma ^k}\left(\frac{d}{2}+\alpha\right)}N+\mu ^{\alpha k}M.
\]
For $t>0$, let $k$ be the integer so that $k\le t<k+1$. Bearing in mind that $0< \mu ,\gamma <1$, we deduce, by straightforward computations, from the preceding inequality
\[
C\|u\|_{L^\infty (\partial D)} \le \mu^{-e^{ct}}N+\mu ^{\alpha t}M.
\]
We end up getting, by taking $e^{ct}=1/\epsilon$,
\[
C\|u\|_{L^\infty (\partial D)} \le e^{c/\epsilon}N+\epsilon ^{\beta}M, \quad 0<\epsilon<1,
\]
which is the expected inequality in (1).

We omit the proof of (2) which is quite similar to that of (1). The only difference is that we have to apply Theorem \ref{theorem1.3} instead of Theorem \ref{theorem1.2}.
\end{proof}

\begin{proposition}\label{proposition1.2}
Let $\omega\Subset D$ and $\tilde{\omega}\Subset D$ be nonempty. There exist $C>0$ and $\beta >0$, only depending on $D$, $\kappa$, $\varkappa$, $\omega$ and $\tilde{\omega}$, so that, for any  $u\in H^1(D)$ satisfying $Lu\in L^2(D)$ and $\epsilon >0$, we have
\begin{align}
&C\|u\|_{L^2(\tilde{\omega})}\leq  \epsilon ^\beta \|u\|_{L^2(D )}+\epsilon^{-1}\left( \|u\|_{L^2(\omega )}+\|Lu\|_{L^2(D)}\right),\label{1.27}
\\
&C\|\nabla u\|_{L^2(\tilde{\omega})}\leq  \epsilon ^\beta \|\nabla u\|_{L^2(D)}+\epsilon^{-1}\left(\|u\|_{L^2(\omega )}+\|Lu\|_{L^2(D)}\right).\label{1.28}
\end{align}
\end{proposition}

\begin{proof}
Since the proof of \eqref{1.27} and \eqref{1.28} are similar, we limit ourselves to the proof of \eqref{1.27}.

Fix $x_0\in \omega$ and $x\in \tilde{\omega}$. Then there exists (see for instance \cite[page 29]{Ch2016}) a sequence of balls $B(x_j,r)$, $r>0$, $j=0,\ldots ,N$, so that
\begin{eqnarray*}
\left\{
\begin{array}{ll}
B(x_0, r)\subset \omega ,
\\
B(x_{j+1},r)\subset B(x_j, 2r),\; j=0,\ldots ,N-1,
\\
x\in B(x_N,r),
\\
B(x_j,3r)\subset \Omega ,\; j=0,\ldots ,N.
\end{array}
\right.
\end{eqnarray*} 
We get then by applying Theorem \ref{theorem1.2}  
\[
\|u\|_{L^2(B(x_j,2r))}\leq C\|u\|_{L^2(B(x_j,3r))}^{1-\gamma}\left(\|u\|_{L^2(B(x_j,r))}+\|Lu\|_{L^2(D)}\right)^\gamma ,\;\; 1\leq j\leq N,
\]
the constants $C>0$ and $0<\gamma <1$ only depend on $D$, $\kappa$ and $\varkappa$.
\\
We proceed as in the proof of Proposition \ref{proposition1.1} in order to obtain
\[
\|u\|_{L^2(B(x_N,2r))}\le C\|u\|_{L^2(B(x_j,3r))}^{1-\gamma^N}\left(\|u\|_{L^2(B(x_0,r))}+\|Lu\|_{L^2(D)}\right)^{\gamma^N} .
\]
Combined with Young's inequality, this estimate yields

\[
C\|u\|_{L^2(B(x_N,2r))}\le \epsilon ^\beta\|u\|_{L^2(D)}+\epsilon^{-1}\left(\|u\|_{L^2(\omega )}+\|Lu\|_{L^2(D)}\right),
\]
where $\beta={\frac{\gamma ^N}{1-\gamma^N}}$.

As $\tilde{\omega}$ is compact, it can be covered by a finite number of balls $B(x_N, r)$, that we denote by 
\[
B(x_N^1, r),\ldots , B(x_N^\ell, r).
\] 
Hence
\[
\|u\|_{L^2(\tilde{\omega})} \le \sum_{j=1}^\ell \|u\|_{L^2(B(x_N^j,r))}.
\]
Whence
\[
C\|u\|_{L^2(\tilde{\omega})}\le \epsilon ^\beta\|u\|_{L^2(D)}+\epsilon^{-1}\left(\|u\|_{L^2(\omega)}+\|Lu\|_{L^2(D)}\right).
\]
\end{proof}

\begin{proposition}\label{proposition1.3}
Let $\mathcal{C}$ be a nonempty open subset of $\partial D$. There exist two constants $C>0$ and $\gamma >0$ and $\omega _0\Subset D$, only depending on $D$, $\kappa$, $\varkappa$ and $\mathcal{C}$,  so that, for any  $u\in H^1(\Omega )$ satisfying $Lu\in L^2(D)$, we have
\[
C\|u\|_{H^1(\omega _0 )}\le \epsilon ^\gamma \|u\|_{H^1(D )} +\epsilon ^{-1}\left(\|u\|_{L^2(\mathcal{C})}+\|\nabla u\|_{L^2(\mathcal{C})}+\|Lu\|_{L^2(D)}\right),\quad \epsilon >0.
\]
\end{proposition}

\begin{proof} 
Let $\tilde{x}\in \mathcal{C}$. Bearing in mind that $D$ is locally on one side of its boundary, we find $x_0$ in the interior of $\mathbb{R}^n\setminus\overline{\Omega}$ sufficiently close to $\tilde{x}$ so that $\rho=\mbox{dist}(x_0,K)<R$, where $K=\overline{B(\tilde{x},R)}\cap \mathcal{C}$. Fix then $r>0$ in order to satisfy $B(x_0,\rho +r)\cap \partial D \subset \mathcal{C}$ and $B(x_0,\rho+\theta r)\cap D\ne \o$, for some $0<\theta <1$.

Define
\[
\psi (x)=\ln \frac{(\rho +r)^2}{|x-x_0|^2}. 
\]
Then 
\[
|\nabla \psi (x)|=\frac{2}{|x-x_0|}\geq \frac{2}{\rho},\quad x\in \overline{D\cap B(x_0,\rho +r)}.
\]

Pick  $\chi \in C_0^\infty (B(x_0,\rho +r))$ satisfying $\chi =1$ on $B(x_0,(1+\theta)r/2)$. Let $u\in H^1(D)$ satisfying $Lu\in L^2(D)$. As we have seen in the proof of Theorem \ref{theorem1.2},
\[
(L(\chi u)) ^2\le (Lu)^2 +(Qu)^2,
\]
with 
\[ 
(Qu)^2\le C\left(u^2+|\nabla u|^2\right)
\] 
and $\mbox{supp}(L(\chi u))\subset U:=B(x_0,\rho+r)\setminus B(x_0, (1+\theta)r/2)$.

We get by applying Theorem \ref{theorem1.1} to $v=\chi u$ in $D\cap B(x_0,\rho +r)$, where $\lambda \ge \lambda _0$ is  fixed and $\tau \geq \tau _0$, 

\begin{align*}
C\int_{B(x_0,\rho+\theta r)\cap D}e^{2\tau \varphi}u^2dx\leq  \int_{U\cap D}&e^{2\tau \varphi}(u^2+|\nabla u|^2)dx+\int_{B(x_0,\rho+r)\cap D} e^{2\tau \varphi}(Lu)^2dx
\\
&+\int_{B(x_0,\rho+r)\cap \partial D}e^{2\tau \varphi}(u^2+|\nabla u|^2)d\sigma .
\end{align*}
But
\[
\varphi (x)=e^{\lambda \ln \frac{(\rho +r)^2}{|x-x_0|^2}}=\frac{(\rho +r)^{2\lambda}}{|x-x_0|^{2\lambda}}.
\]
Whence
\begin{align}
&Ce^{2\tau \varphi_0}\int_{B(x_0,\rho+\theta r)\cap D}u^2dx\le e^{2\tau \varphi_0}\int_{U\cap D}(u^2+|\nabla u|^2)dx \label{1.29}
\\
&\qquad+e^{2\tau \varphi _2}\int_{B(x_0,\rho+r)}(Lu)^2dx +e^{2\tau \varphi _2}\int_{B(x_0,\rho+r)\cap \mathcal{C}}(u^2+|\nabla u|^2)d\sigma ,\nonumber
\end{align}
where
\[
\varphi _0=\frac{(\rho +r)^{2\lambda}}{(\rho +\theta r)^{2\lambda}},\quad \varphi _1=\frac{(\rho +r)^{2\lambda}}{(\rho +(1+\theta)r/2)^{2\lambda}},\quad \varphi _2=\frac{(\rho +r)^{2\lambda}}{\rho ^{2\lambda}}.
\]
Let
\[
\alpha= \frac{r(1-\theta)\lambda (\rho +r)^{2\lambda}}{(\rho +(1+\theta)r/2)^{2\lambda +1}}\quad\mbox{and}\quad \beta =\frac{2r\lambda \theta(\rho +r)^{2\lambda}}{\rho ^{2\lambda +1}}.
\]
Elementary computations show 
\[
\varphi _0-\varphi _1 \ge \alpha \quad \mbox{and}\quad \varphi _2-\varphi _0\leq \beta .
\]
These inequalities in  \eqref{1.29} yield
\begin{align}
&C\int_{B(x_0,r+\theta r)\cap D}u^2dx\leq e^{-\alpha \tau }\int_{U}(u^2+|\nabla u|^2)dx\label{1.30}
\\
&\qquad +e^{\beta \tau }\int_{B(x_0,\rho+r)\cap D}(Lu)^2dx +e^{\beta \tau }\int_{B(x_0,\rho+r)\cap \partial D}(u^2+|\nabla u|^2)d\sigma .\nonumber
\end{align}
Let $\omega_0\Subset \omega_1\Subset B(x_0,r+\theta r)\cap D$. We can then mimic the proof of Caccioppoli's inequality in Lemma \ref{lemma1.1} in order to obtain
\begin{equation}\label{1.31}
C\int_{\omega _0}|\nabla u|^2dx\le \int_{\omega _1}u^2dx+\int_{\omega_1}(Lu)^2dx.
\end{equation}
Hence \eqref{1.31} in \eqref{1.30} yields
\begin{align*}
C\int_{\omega_0}(u^2+|\nabla u|^2)dx\le e^{-\alpha \tau }&\int_D(u^2+|\nabla u|^2)dx
\\
& +e^{\beta \tau }\int_D(Lu)^2dx +e^{\beta \tau }\int_{\mathcal{C} }(u^2+|\nabla u|^2)d\sigma .
\end{align*}
Again, we complete the proof similarly to that of Theorem \ref{theorem1.2}.
\end{proof}

We shall need hereafter the following lemma.

\begin{lemma}\label{lemma1.3}
$($\cite{Ch2016}$)$ There exists a constant $C>0$, only depending on $D$ and $\kappa$, so that, for any $u\in H^1(D)$ with $Lu\in L^2(D)$, we have
\begin{equation}\label{1.32}
C\| u\|_{H^1(D)} \le  \|Lu\|_{L^2(D)}+ \|u\|_{H^{1/2}(\partial D)}.
\end{equation}
\end{lemma}

\begin{theorem}\label{theorem1.4}
Let $\mathcal{C}$ be a nonempty open subset of $\partial D$ and let $0<\alpha \le 1$. There exist $C>0$, $c >0$ and $\beta >0$, only depending on $D$, $\kappa$, $\varkappa$, $\alpha$ and $\mathcal{C}$, so that, for any $u\in C^{1,\alpha}(\overline{D} )$ satisfying $Lu\in L^2(D)$ and $0<\epsilon <1$,  we have
\begin{equation}\label{1.33}
C\|u\|_{H^1(D)}\le \epsilon ^\beta \|u\|_{C^{1,\alpha}(\overline{D} )}+e^{c/\epsilon}\left(\|u\|_{L^2(\mathcal{C})}+\|\nabla u\|_{L^2(\mathcal{C})}+\|Lu\|_{L^2(D)}\right).
\end{equation}
\end{theorem}

\begin{proof}
Henceforward, the generic constants $C>0$ and $c>0$ can only depend on $D$, $\kappa$, $\varkappa$, $\alpha$ and $\mathcal{C}$.

Let $u\in C^{1,\alpha}(\overline{D} )$ satisfying $Lu\in L^2(D)$. Then, from Lemma \ref{lemma1.3}, we have 
\[
C\| u\|_{H^1(D)}\le \|Lu\|_{L^2(D)}+ \|u\|_{H^{1/2}(\partial D)}.
\]
By Proposition \ref{proposition1.1} and noting that $W^{1,\infty}(\partial D )$ is continuously embedded in $H^{1/2}(\partial D)$, there exist $\beta >0$ and $\omega \Subset  D$ so that, for any $0<\epsilon <1$,
\begin{equation}\label{1.34}
C\|u\|_{H^1(D)}\le \epsilon ^\beta \|u\|_{C^{1,\alpha}(\overline{D} )}+e^{c/\epsilon}\left(\|u\|_{H^1(\omega )}+\|Lu\|_{L^2(D)}\right).
\end{equation}

On the other hand, by Proposition \ref{proposition1.3}, there exist $\omega _0\Subset \Omega$ and $\gamma >0$ so that, for any $\epsilon_1 >0$,
\begin{equation}\label{1.35}
C\|u\|_{H^1(\omega _0)}\le \epsilon_1 ^\gamma \|u\|_{C^{1,\alpha}(\overline{D} )} +\epsilon_1 ^{-1}\left(\|u\|_{L^2(\mathcal{C})}+\|\nabla u\|_{L^2(\mathcal{C})}+\|Lu\|_{L^2(D)}\right).
\end{equation}
But, by Proposition \ref{proposition1.2}, there is $\delta >0$ such that, for any $\epsilon_2 >0$,
 \begin{equation}\label{1.36}
 C\|u\|_{H^1(\omega )}\leq \epsilon_2 ^\delta \|u\|_{C^{1,\alpha}(\overline{D} )}+\epsilon_2 ^{-1}\left(\|u\|_{H^1(\omega _0)}+\|Lu\|_{L^2(D)}\right).
 \end{equation}
Estimate \eqref{1.35} in \eqref{1.36} gives
\begin{align*}
C\|u\|_{H^1(\omega )}&\le (\epsilon_2 ^\delta +\epsilon_2 ^{-1} \epsilon_1 ^\gamma) \|u\|_{C^{1,\alpha}(\overline{D} )} 
\\
&\quad +\epsilon_1 ^{-1}\epsilon_2 ^{-1}\left(\|u\|_{L^2(\mathcal{C})}+\|\nabla u\|_{L^2(\mathcal{C})}+\|Lu\|_{L^2(D)}\right)+\epsilon_2 ^{-1}\|Lu\|_{L^2(D)}.
\end{align*}
The choice of $\epsilon _1=\epsilon_2^{\frac{\gamma +1}{\delta}}$ in this estimate yields, where $\varrho=\frac{\gamma +\delta +1}{\delta}$,
\begin{align*}
C\|u\|_{H^1(\omega )}&\le \epsilon_2 ^\delta \|u\|_{C^{1,\alpha}(\overline{D} )} 
\\
&\quad +\epsilon_2 ^{-\varrho}\left(\|u\|_{L^2(\mathcal{C})}
+\|\nabla u\|_{L^2(\mathcal{C})^n}+\|Lu\|_{L^2(D)}\right)+\epsilon_2 ^{-1}\|Lu\|_{L^2(D)},
\end{align*}
which, in combination with \eqref{1.34}, implies
\begin{align*}
&C\|u\|_{H^1(D)}\le (\epsilon ^\beta +\epsilon_2 ^\delta e^{c/\epsilon})\|u\|_{C^{1,\alpha}(\overline{D} )}
\\
&\qquad +\epsilon _2^{-\varrho} e^{c/\epsilon}\left(\|u\|_{L^2(\mathcal{C})}+\|\nabla u\|_{L^2(\mathcal{C})}+\|Lu\|_{L^2(D)}\right)+(\epsilon_2 ^{-1}+1)e^{c/\epsilon}\|Lu\|_{L^2(D)}.
\end{align*}
Therefore,
\begin{align*}
C\|u\|_{H^1(D)}&\le (\epsilon ^\beta +\epsilon_2 ^\delta e^{c/\epsilon})\|u\|_{C^{1,\alpha}(\overline{D} )}
\\
&+\left(\epsilon _2^{-\varrho}+\epsilon_2 ^{-1}+1\right) e^{c/\epsilon}\left(\|u\|_{L^2(\mathcal{C})}+\|\nabla u\|_{L^2(\mathcal{C})}+\|Lu\|_{L^2(D)}\right).
\end{align*}
We end up getting the expected inequality by taking $\epsilon _2=e^{-2c/(\epsilon \delta)}$.
\end{proof}

As a first consequence of the preceding theorem, we have the uniqueness of continuation of solutions from Cauchy data on $\mathcal{C}$.

\begin{corollary}\label{corollary1+}
Let $\mathcal{C}$ be a nonempty open subset of $\partial D$ and let $0<\alpha \le 1$. If $u\in C^{1,\alpha}(\overline{D} )$ satisfies $Lu=0$ in $D$ and $u=\partial_\nu u=0$ on $\mathcal{C}$, then $u=0$.
\end{corollary}

Define, for $c>0$ and $\beta >0$, the function $\Phi_{c,\beta}$ by
\[
\Phi_{c,\beta}(\rho)=\left\{ \begin{array}{ll} 0\quad &\mbox{if}\; \rho =0,\\  |\ln \rho |^{-\beta}\quad &\mbox{if}\; 0<\rho < e^{-c}, \\ \rho &\mbox{if}\; \rho \ge e^{-c}. \end{array}\right.
\]

Let $\mathcal{C}$ be a nonempty open subset of $\partial D$ and  set
\[
\mathfrak{D}(u)=\|u\|_{L^2(\mathcal{C})}+\|\nabla u\|_{L^2(\mathcal{C})}+\|Lu\|_{L^2(D)}.
\]

\begin{corollary}\label{corollary2+}
Let  $0<\alpha \le 1$. There exist $C>0$, $c >0$ and $\beta >0$, only depending on $D$, $\kappa$, $\varkappa$, $\alpha$ and $\mathcal{C}$, so that, for any $0\ne u\in C^{1,\alpha}(\overline{D} )$ satisfying $Lu\in L^2(D)$, we have
\begin{equation}\label{+2}
C\|u\|_{H^1(D)}\le \|u\|_{C^{1,\alpha}(\overline{D} )}\Phi_{c,\beta}\left(\frac{\mathfrak{D}(u)}{\|u\|_{C^{1,\alpha}(\overline{D} )}}  \right).
\end{equation}
\end{corollary}

\begin{proof}
Pick $u\in C^{1,\alpha}(\overline{D} )$ satisfying $Lu\in L^2(D)$ and $u\ne 0$. For sake of simplicity, we use in this proof the following temporary notations
\begin{align*}
&\mathfrak{a}=\|u\|_{C^{1,\alpha}(\overline{D} )},
\\
&\mathfrak{b}=\mathfrak{D}(u).
\end{align*}
Since $u\ne 0$, it follows from Corollary \ref{corollary1+} that $\mathfrak{b}\ne 0$. Then, according to Theorem \ref{theorem1.4}, we have
\begin{equation}\label{+1}
C\frac{\|u\|_{H^1(D)}}{\mathfrak{a}}\le \epsilon^{\beta}+\frac{\mathfrak{b}}{\mathfrak{a}}e^{c\epsilon},\quad 0<\epsilon <1.
\end{equation}
Assume first that $\mathfrak{b}/\mathfrak{a}< e^{-c}$. Since the function $\epsilon \rightarrow \epsilon^\beta e^{-c/\epsilon}$, extended by continuity at $\epsilon=0$, is nondecreasing, we may find $0<\epsilon_0 <1$ so that $\epsilon_0^\beta e^{-c/\epsilon_0}=\mathfrak{b}/\mathfrak{a}$. Therefore
\[
\frac{\mathfrak{a}}{\mathfrak{b}}=\frac{1}{\epsilon_0^\beta}e^{c/\epsilon_0}\le e^{(c+\beta)/\epsilon_0}
\]
and hence
\[
\epsilon_0 \le (c+\beta )|\ln (\mathfrak{b}/\mathfrak{a})|^{-1}.
\]
Then $\epsilon =\epsilon_0$ in \eqref{+1} yields the expected inequality in this case. We end the proof by noting that \eqref{+2} is obvious satisfied if $\mathfrak{b}/\mathfrak{a}\ge e^{-c}$. Indeed, in that case we have $C\|u\|_{H^1(D)}\le \mathfrak{a}\le e^c\mathfrak{b}=e^c\mathfrak{a}\Phi_{c,\beta}(\mathfrak{b}/\mathfrak{a})$.
\end{proof}

\end{document}